\documentclass{au}
 \presentedby{\dots}
 \received{\dots}{\dots}

\usepackage{amsmath,amssymb,url,mathrsfs}
\usepackage{enumerate}
\usepackage{tikz}
\usetikzlibrary{calc}
\usepackage{scalefnt}
\numberwithin{equation}{section}

\theoremstyle{plain}
\newtheorem{theorem}{Theorem}[section]
\newtheorem{lemma}[theorem]{Lemma}

\newtheorem{corollary}[theorem]{Corollary}

\theoremstyle{definition}

\newtheorem{example}[theorem]{Example}
\newtheorem*{remarks}{Remarks}

\newcommand{\suchthat}{\ensuremath{\mid}}  
\newcommand{\defeq}{\ensuremath{=}}        
\newcommand{\betai}{\ensuremath{\beta^i}}  
\newcommand{\betaj}{\ensuremath{\beta^j}}  
\newcommand{\uacalc}{{\small UAC}alc} 

\newcommand{\Palfy}{P\'alfy}
\newcommand{\Pudlak}{Pudl\'ak}
\newcommand{\PP}{P\'alfy-Pudl\'ak}

\newcommand{\<}{\ensuremath{\langle}}
\renewcommand{\>}{\ensuremath{\rangle}}

\newcommand{\bA}{\ensuremath{\mathbf{A}}}
\newcommand{\bB}{\ensuremath{\mathbf{B}}}

\newcommand{\sB}{\ensuremath{\mathcal{B}}}

\newcommand{\sT}{\ensuremath{\mathscr{T}}}
\newcommand{\sI}{\ensuremath{\mathcal{I}}}

\DeclareMathOperator{\id}{id}
\DeclareMathOperator{\Eq}{Eq}
\DeclareMathOperator{\Cg}{Cg}
\DeclareMathOperator{\Con}{Con}
\DeclareMathOperator{\Pol}{Pol}
\DeclareMathOperator{\Clo}{Clo}

\newcommand{\htheta}{\ensuremath{\widehat{\theta}}}
\newcommand{\supi}{\ensuremath{^{i}}}
\newcommand{\supj}{\ensuremath{^{j}}}
\renewcommand{\leq}{\ensuremath{\leqslant}}

\renewcommand{\geq}{\ensuremath{\geqslant}}

\renewcommand{\Join}{\ensuremath{\bigvee}}
\newcommand{\resB}{\ensuremath{|_{_B}}}

\newcommand{\eps}{\ensuremath{\varepsilon}}
\newcommand{\hatmap}{\ensuremath{\widehat{\phantom{x}}}} 
\newcommand{\one}{\ensuremath{\mathbf{1}}}
\newcommand{\two}{\ensuremath{\mathbf{2}}}
\newcommand{\three}{\ensuremath{\mathbf{3}}}
\newcommand{\tbeta}{\ensuremath{\widetilde{\beta}}}

\newcommand{\hbeta}{\ensuremath{\widehat{\beta}}}

\newcommand{\CE}{\ensuremath{\sB_r^n}}

\newcommand{\GAP}{{\small GAP}} 
\newcommand{\dotsize}{1.0pt}

\begin{document}


\title[Expansions of finite algebras]{Expansions of finite algebras and their\\ 
congruence lattices}%

\author[W. DeMeo]{William DeMeo} 
\email{williamdemeo@gmail.com}
\urladdr{http://www.math.sc.edu/~demeow}
\address{Department of Mathematics\\
University of South Carolina\\Columbia 29208\\USA}


\dedicatory{This article is dedicated to Ralph Freese and Bill
  Lampe.}

\subjclass[2010]{Primary: 08A30; Secondary: 08A60, 06B10.} 

\keywords{congruence lattice, finite algebra, finite lattice
  representations}


\begin{abstract}
In this paper we present a novel approach to the construction of new finite
algebras and describe the congruence lattices of these algebras. Given a finite
algebra $\<B_0, \dots\>$, let $B_1, B_2, \dots, B_K$ be sets that either intersect
$B_0$ or intersect each other at certain points.
We construct an  {\it overalgebra} $\<A, F_A\>$, by which we mean an expansion
of $\<B_0, \dots\>$ with universe  $A = B_0 \cup B_1 \cup \cdots \cup B_K$,  and
a certain set $F_A$ of unary operations that includes mappings $e_i$ satisfying
$e_i^2 = e_i$ and $e_i(A) = B_i$, for $0\leq i \leq K$.
We explore two such constructions and prove results about the shape of
the new congruence lattices $\Con \<A, F_A\>$ that result. Thus, descriptions of
some new classes of finitely representable lattices is one contribution of this
paper. Another, perhaps more significant contribution is the announcement of a
novel approach to the discovery of new classes of representable lattices, the
full potential of which we have only begun to explore.
\end{abstract}

\maketitle

\section{Introduction}
A lattice $L$ is called \emph{algebraic} if it is complete and if every element
$x \in L$ is the supremum of the compact elements below $x$.
Equivalently, a lattice is algebraic if and only if it is the congruence lattice
of an algebra.  (The fact that every algebraic lattice is the congruence lattice of an
algebra was proved by Gr\"{a}tzer and Schmidt
in~\cite{GratzerSchmidt:1963}.)  An important and long-standing open problem in
universal algebra is to characterize those lattices that are isomorphic to congruence
lattices of \emph{finite} algebras.  We call such lattices \emph{finitely representable}.
Until this problem is resolved, our understanding of finite algebras is
incomplete, since, given an arbitrary finite algebra, we cannot say
whether there are any restrictions on the shape of its congruence lattice.
If we find a finite lattice which does not occur as the congruence lattice of a
finite algebra (as many suspect we will), then we can finally declare that such
restrictions do exist.

The main contribution of this paper is the description and analysis of a
new procedure for generating finite lattices which are, by
construction, finitely representable.
Roughly speaking, we start with an arbitrary finite algebra $\bB \defeq  \<B,
\dots\>$, with known congruence lattice $\Con\bB$, and we let $B_1, B_2, \dots,
B_K$ be sets that either intersect $B$ or intersect each other at certain points.
The choice of intersection points plays an important role which we will describe
in detail later.  We then construct an  {\it overalgebra} $\bA\defeq \<A,
F_A\>$, by which we mean an expansion of $\bB$ with universe $A = B \cup B_1
\cup \cdots \cup B_K$,  and a certain set $F_A$ of unary operations that
includes idempotent mappings $e$ and $e_i$ satisfying $e(A) = B$ and $e_i(A) =
B_i$.

Given our interest in the problem mentioned above, the important consequence of
this procedure is the new (finitely representable) lattice $\Con\bA$ that
it produces.  The shape of this lattice is, of course, determined by
the shape of $\Con\bB$, the choice of intersection points of the $B_i$, and the
unary operations chosen for inclusion in $F_A$.  In this paper, we
describe two constructions of this type and prove some results
about the shape of the congruence lattices of the resulting overalgebras.
However, it is likely that we have barely scratched the surface of useful
constructions that are possible with this approach.

Before giving an overview of the paper, we describe the seminal example that
provided the impetus for this work.  In the spring of 2011, while our research
seminar was mainly focused on the finite congruence lattice representation
problem, we were visited by Peter Jipsen who initiated the ambitious project of
cataloging every small finite lattice $L$ for which there is a known finite
algebra $\bA$ with $\Con\bA\cong L$.  Before long we had identified such finite
representations for all lattices of order seven or less, except for the two
lattices appearing in Figure~\ref{fig:sevens}.
\begin{figure}[h!]
  \centering

    \begin{tikzpicture}[scale=.7]

      \node (01) at (0,1)  [draw, circle, inner sep=\dotsize] {};
      \foreach \j in {0,2}
      { \node (1\j) at (1,\j)  [draw, circle, inner sep=\dotsize] {};}

      \foreach \j in {1,3}
      { \node (2\j) at (2,\j)  [draw, circle, inner sep=\dotsize] {};}
      { \node (32) at (3,2)  [draw, circle, inner sep=\dotsize] {};}
      \draw[semithick] (10) to (01) to (12) to (23) to (32) to (21) to (10) (21) to (12);
      { \node (m11) at (-1,1)  [draw, circle, inner sep=\dotsize] {};}
      \draw[semithick] (10) to (m11) to (23);


      \foreach \j in {0,3}
      { \node (7\j) at (6.5,\j)  [draw, circle, inner sep=\dotsize] {};}
      \node (71) at (7,1.5)  [draw, circle, inner sep=\dotsize] {};
      \node (61) at (6,1.5)  [draw, circle, inner sep=\dotsize] {};
      \node (51) at (5,1.5)  [draw, circle, inner sep=\dotsize] {};
      \foreach \j in {1,2}
      { \node (8\j) at (7.8,\j)  [draw, circle, inner sep=\dotsize] {};}
      \draw[semithick] (70) to (51) to (73) to (61) to (70) to (71) to (73) to
      (82) to (81) to (70);

    \end{tikzpicture}

  \caption{Lattices of order 7 with no obvious finite algebraic representation.}
  \label{fig:sevens}
\end{figure}
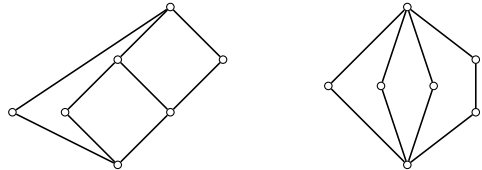
Ralph Freese then discovered a way to construct an algebra
that has the second of these as its congruence lattice. The idea
is to start with an algebra $\bB = \<B, \dots\>$ having congruence lattice $\Con \bB
\cong M_4$, expand the universe to a larger set, $A = B\cup B_1 \cup B_2$, and
then define the right set $F_A$ of operations on $A$ so that the congruence lattice
of $\bA = \< A, F_F\>$ will be an $M_4$ with one atom 
``doubled''---that is,
$\Con\bA$ will be the second lattice in Figure~\ref{fig:sevens}.

In this paper we formalize this approach and extend it in two ways.  The first
is a generalization of the original overalgebra construction.
The second is a construction based on one suggested by Bill Lampe that
addresses a basic limitation of the original procedure.  For each of these
constructions we prove results that describe the congruence
lattices of the resulting overalgebras.

Here is a brief outline of the remaining sections of the paper:
In Section~\ref{sec:residuation-lemma} we prove a lemma that simplifies
the analysis of the structure of the newly enlarged congruence lattice and
its relation to the original congruence lattice.
In Section~\ref{sec:overalgebras} we define {\it overalgebra} and in
Section~\ref{sec:overalgebras-i} we give a formal description of
the first overalgebra construction mentioned above.  We then describe Freese's
example in detail before proving some general results about
the congruence lattices of such overalgebras.
We conclude Section~\ref{sec:overalgebras-i} with an example demonstrating the
utility of the first overalgebra construction.
Section~\ref{sec:overalgebras-ii}
presents a second overalgebra construction which overcomes a basic limitation
of the first.
The last section discusses the impact that our results
have on the main problem---the finite congruence lattice representation problem---as well as the inherent limitations of this approach.

\section{Residuation lemma} 
\label{sec:residuation-lemma}
Let $\bA \defeq  \<A, \dots \>$ be an algebra with congruence lattice $\Con\<A, \dots \>$.
Recall that a \emph{clone} on a non-void set $A$ is a set of operations on $A$
that contains the projection operations and is closed under compositions.
The \emph{clone of term operations} of the algebra $\bA$, denoted by $\Clo (\bA)$,
is the smallest clone on $A$ containing the basic operations of $\bA$.
The \emph{clone of polynomial operations} of $\bA$, denoted by
$\Pol(\bA)$, is the clone generated by the basic operations
of $\bA$ and the constant unary maps on $A$. The set of $n$-ary members of
$\Pol(\bA)$ is denoted by $\Pol_n(\bA)$.
It is not hard to show that
$\Con\<A, \dots \>  = \Con \<A, \Clo (\bA)\> = \Con \<A, \Pol (\bA)\>= 
\Con \<A, \Pol_1 (\bA)\>$; see~\cite[Theorem~4.18]{alvi:1987}.

Suppose $e \in \Pol_1(\bA)$ is a unary polynomial satisfying $e^2(x) = e(x)$ for
all $x\in A$.  Define $B=e(A)$ and
$F_B = \{ef\resB \suchthat f\in \Pol_1(\bA)\}$, and consider the
unary algebra $\bB= \<B, F_B\>$. 
(In the definition of  $F_B$, we could have used
$\Pol(\bA)$ instead of $\Pol_1(\bA)$, and then our discussion would not be
limited to unary algebras.  However, we are mainly concerned with
congruence lattices, so we lose nothing by restricting the scope in this way.)

P\'eter \Palfy\ and
Pavel \Pudlak\
prove in~\cite[Lemma~1]{Palfy:1980} that
the restriction mapping $\resB$, defined on $\Con\bA$ by
$\alpha\resB = \alpha \cap B^2$, is a lattice epimorphism of $\Con\bA$ onto $\Con\bB$.
In~\cite{McKenzie:1983}, Ralph McKenzie
developed the foundations of what would become tame congruence theory, and
the \PP\ lemma played a seminal role in this development.  In his presentation
of the lemma, McKenzie introduced the mapping $\hatmap$ defined on $\Con\bB$ as follows:
\[
\widehat{\beta} = \{(x,y) \in A^2 \suchthat  \text{ for all } f\in \Pol(\bA), \,
(ef(x), ef(y))\in \beta\}.
\]
Throughout this paper, we use a definition of $\hatmap$ that is effectively
the same.  Whenever $\bA = \< A, \dots\>$
 and $\bB = \< B, \dots\>$ are algebras with $B = e(A)$ for some
$e^2 = e \in \Pol_1(\bA)$,  we take the map
$\hatmap 
\colon \Con\bB \rightarrow \Con\bA$ to mean
\begin{equation}
  \label{eq:hatmap}
\widehat{\beta} = \{(x,y) \in A^2 \suchthat \text{ for all }
f\in \Pol_1(\bA), \, (ef(x), ef(y))\in \beta \}.
\end{equation}
It is not hard to see that $\hatmap$ maps $\Con\bB$ into $\Con\bA$.  For
example, if $(x,y) \in \widehat{\beta}$ and $g\in \Pol_1(\bA)$, then for all $f\in \Pol_1(\bA)$ we
have $(efg(x),efg(y)) \in \beta$, so $(g(x),g(y))\in \widehat{\beta}$.

For each $\beta \in \Con\bB$, let $\beta^* = \Cg^\bA(\beta)$.  That is,
$^* 
\colon \Con\bB \rightarrow \Con\bA$ is
the congruence generation operator restricted to the set $\Con\bB$.
The following lemma concerns the three mappings, $\resB$, $\hatmap$, and $^*$.
The third statement of the lemma, which follows from the first two,
will be useful in the later sections of the paper.

\begin{lemma}\ 
\begin{enumerate}[\rm(i)] 
  \item \label{item:residlemma-i} $^*\colon \Con\bB \rightarrow \Con\bA$ is a residuated mapping with
    residual $\resB$.
  \item \label{item:residlemma-ii} $\resB \colon  \Con\bA \rightarrow \Con\bB$ is a residuated mapping with
    residual $\hatmap$.
\item \label{item:residlemma-iii} For all $\alpha \in \Con\bA$, for all $\beta \in \Con\bB$,
\[
\beta = \alpha\resB \quad \Longleftrightarrow  \quad
\beta^* \leq \alpha \leq \widehat{\beta}.
\]
In particular,
$\beta^*\resB = \beta = \widehat{\beta}\resB$.
  \end{enumerate}
\end{lemma}
\begin{proof}
  We first recall the definition of {\it residuated mapping}.  If $X$ and $Y$
  are partially ordered sets, and if
$f\colon  X \rightarrow Y$ and
$g\colon  Y \rightarrow X$ are order preserving maps, then the following are
equivalent:
\begin{enumerate}[(a)]
\item $f\colon  X \rightarrow Y$ is a {\it residuated mapping} with {\it residual}
$g\colon  Y \rightarrow X$;
\item 
$(\forall  x\in X)(\forall y\in Y) \, f(x) \leq y \iff x \leq g(y)$;
\item $g\circ f \geq \id_X$ and $f\circ g \leq \id_Y$,
\end{enumerate}
where $\id_S$ denotes the identity map on the set $S$.
The definition says that for each $y\in Y$ there is a unique
$x\in X$ that is maximal with respect to the property $f(x) \leq y$, and the
maximum is given by $x = g(y)$.
Thus, (\ref{item:residlemma-i}) is equivalent to:
$\forall  \alpha \in \Con\bA,\, \forall  \beta \in \Con\bB$,
\begin{equation}
  \label{eq:1}
\beta^* \leq \alpha \iff \beta \leq \alpha\resB.
\end{equation}
This is easily verified, as follows:  If
$\beta^* \leq \alpha$ and $(x,y)\in \beta$, then
$(x,y) \in \beta^* \leq \alpha$
and $(x,y) \in B^2$, so $(x,y)\in
\alpha\resB$.  If $\beta \leq \alpha\resB$ then
$\beta^* \leq (\alpha\resB)^* \leq \Cg^\bA(\alpha) = \alpha$.

Statement (\ref{item:residlemma-ii}) is equivalent to:
$\forall \alpha \in \Con\bA, \, \forall \beta \in \Con\bB$,
\begin{equation}
  \label{eq:2}
\alpha\resB\leq \beta \iff \alpha \leq \widehat{\beta}.
\end{equation}
This is also easy to check.  For, suppose
$\alpha\resB\leq \beta$ and $(x,y)\in \alpha$. Then $(ef(x), ef(y)) \in \alpha$
for all $f \in \Pol_1(\bA)$ and $(ef(x), ef(y)) \in B^2$, therefore,
$(ef(x), ef(y)) \in \alpha\resB \leq \beta$, so $(x,y) \in \widehat{\beta}$.
Suppose $\alpha \leq \widehat{\beta}$ and $(x,y) \in \alpha\resB$.
Then $(x,y) \in \alpha \leq  \widehat{\beta}$, so
$(ef(x), ef(y)) \in \beta$ for all $f\in \Pol_1(\bA)$, including $f=\id_A$, so
$(e(x), e(y)) \in \beta$. But $(x, y) \in B^2$, so $(x, y) = (e(x), e(y)) \in
\beta$.

Combining~(\ref{eq:1}) and~(\ref{eq:2}), we obtain statement (\ref{item:residlemma-iii}) of the lemma.
\end{proof}

The lemma above was inspired by the two approaches to
proving~\cite[Lemma~1]{Palfy:1980}.  In the original paper $^*$ is used, while
McKenzie uses the $\hatmap$ operator.  Both $\beta^*$ and
$\widehat{\beta}$ are mapped onto $\beta$ by the restriction map $\resB$, so
the restriction map is indeed onto $\Con\bB$.
However, our lemma emphasizes the fact that the interval
\[
[\beta^*, \widehat{\beta}] =
\{\alpha \in \Con\bA \suchthat \beta^* \leq \alpha \leq \widehat{\beta}\}
\]
is precisely the set of congruences for
which $\alpha\resB = \beta$.  In other words, the
$\resB$-inverse image of $\beta$ is
$[\beta^*, \widehat{\beta}]$.
This fact plays a central role in the
theory developed in this  paper.
For the sake of completeness, we conclude this section by
verifying that~\cite[Lemma~1]{Palfy:1980} can be obtained from the lemma above.
\begin{corollary}
[cf.~{\cite[Lemma 1]{Palfy:1980}}]
The mapping  $\resB \colon  \Con\bA \rightarrow \Con\bB$ is onto and preserves meets and joins.
\end{corollary}
\begin{proof}
  Given $\beta\in \Con\bB$, each $\theta\in \Con\bA$ in the interval $[\beta^*,
  \widehat{\beta}]$ is mapped to $\theta\resB = \beta$, so $\resB$ is clearly
  onto.  That $\resB$ preserves meets is obvious, so we just check that $\resB$ is
%
%
  join preserving.  Since $\resB$ is order preserving, we have, for all 
$S \subseteq \Con\bA$,
\[
\Join \theta\resB \leq \bigl(\Join \theta\bigr)\resB,
\]
where joins are over all $\theta \in S$.  The opposite inequality follows
from~(\ref{eq:2}) above. Indeed, by~(\ref{eq:2}) we have
\[
\bigl(\Join \theta\bigr)\resB \leq \Join \theta\resB
\quad \Longleftrightarrow \quad
\Join \theta \leq \bigl(\Join \theta\resB \bigr)\!\widehat{\phantom{X}}
\]
and the last inequality holds by another application of~(\ref{eq:2}): if $\eta \in S$, then
\[
\eta \leq \bigl(\Join \theta\resB\bigr)\!\widehat{\phantom{X}}
\quad \Longleftrightarrow \quad
\eta\resB \leq \Join \theta\resB.
\]

\end{proof}

\section{Overalgebras}
\label{sec:overalgebras}
In the previous section, we started with an algebra $\bA$ and
considered a subreduct $\bB$ with universe $B = e(A)$, the image of an
idempotent unary polynomial of $\bA$.  In this section, we start with a
fixed finite algebra $\bB = \<B, \dots \>$ and consider various ways to
construct an \emph{overalgebra}, that is, an algebra $\bA= \<A, F_A\>$ having
$\bB$ as a subreduct where $B = e(A)$ for some $e^2 = e \in F_A$.
Beginning with a specific finite algebra $\bB$, our goal is to understand what
congruence lattices $\Con\bA$ can be built up from
$\Con\bB$ by expanding the algebra $\bB$ in this way.

\subsection{Overalgebras I}
\label{sec:overalgebras-i}
Let $B$ be a finite set and let $F$ be a set of unary maps that take $B$
into itself. Consider the algebra $\bB = \<B, F\>$ with universe $B$ and basic operations $F$.
Let $T$ denote a finite sequence $t_1, t_2, \dots, t_K$ of elements of $B$ (possibly with repetitions).
For each $1\leq i \leq K$, let $B_i$ be a set of the same cardinality as $B$
and intersecting $B$ at the $i$-th point of the sequence $T$; that is,
$B_i\cap B = \{t_i\}$.  Assume also that $B_i\cap B_j = \emptyset$ if $t_i\neq t_j$, and
$B_i\cap B_j = \{t_i\}$ if $t_i= t_j$.  Occasionally, for notational convenience, we
use the label $B_0\defeq B$, and for extra clarity we may write
$F_B$ in place of $F$.

For each $1\leq i \leq K$, let $\pi_i\colon  B\rightarrow B_i$ be a bijection that leaves
$t_i$ fixed.  That is $\pi_i(t_i) = t_i$ and otherwise $\pi_i$ is an arbitrary
(but fixed) bijection.  Let $\pi_0 = \id_{B}$, the identity map on $B$.  For
$b\in B$ we often use the label $b^i$ to denote $\pi_i(b)$.  The map $\pi_i$ and
the operations $F$ induce a set $F_i$ of unary operations on $B_i$ as follows:
to each $f\in F$ corresponds the operation $f^{\pi_i} \colon  B_i \rightarrow B_i$
defined by $f^{\pi_i} = \pi_i f \pi_i^{-1}$.
It is easy to see that $\bB_i \defeq  \<B_i, F_i\>$
and $\bB = \<B, F\>$ are isomorphic  algebras.
(Indeed, $\pi_i$ is a bijection of the universes that respects the
interpretation of the basic operations: 
$\pi_if(b) = \pi_i f(\pi_i^{-1}\pi_i b) = f^{\pi_i}(\pi_i b)$.)

Let $A = \bigcup_{i=0}^K B_i$ and for $0\leq k \leq K$ define the maps
$e_k\colon A\rightarrow A$ as follows:
\[
e_0(x) = \pi^{-1}_{\iota x}(x) \quad\text{and}\quad
e_k(x)=\pi_k e_0(x),
\]
where $\iota x$ denotes the smallest index $k$ such that $x\in
B_k$. 
Thus, $e_0$ acts as the identity on 
$B_0$ $(= B)$ and maps each
$b^i\in B_i$ onto the corresponding element $b\in B$.
Similarly, $e_k$ acts as the identity on $B_k$ and maps each
$b^i\in B_i$ onto the corresponding element $b^k\in B_k$.

The next set of maps that we define on $A$ will be based on a given partition of the
elements of $T$.  Actually, since $T$ is a sequence, possibly with repetitions,
we must consider instead a partition  of
the \emph{indices} of $T$.  Let $\sT = |\sT_1|\sT_2|\dots|\sT_N|$ denote this
partition of the index set $\{1, 2, \dots, K\}$.
We now define maps $s_n \colon  A\rightarrow A$
based on the partition $\sT$, as follows: for each $1\leq n\leq N$, let
\begin{equation}
  \label{fn:s_n}
s_n(x) =
 \begin{cases}
   t_i, & \text{if $x \in B_i$ for some $i \in \sT_n$, }\\
   x, & \text{otherwise}.
 \end{cases}
\end{equation}
(Recall, $B_i$ intersects $B$ at the point $t_i$. If the index of this point
belongs to block $\sT_n$, then $s_n$ collapses $B_i$ onto $t_i$; otherwise,
$s_n$ acts as the identity on $B_i$.  A typical special case of our
construction takes the partition $\sT$ to be the trivial
partition consisting of a single block $\sT_1 = 1,2,\dots,K$.  In this case
$N=1$, and the map $s_1$ acts as the identity on $B$ and collapses each set $B_i$
onto the point at which it intersects $B$.)

Finally, letting
\[
F_A \defeq  \{f e_0 \suchthat f\in F_B\} \cup \{e_k \suchthat 0\leq k \leq K\} \cup \{s_n \suchthat 1 \leq n
\leq N\},
\]
we define the algebra $\bA \defeq  \< A, F_A\>$, which we call
an \emph{overalgebra of} $\bB$.
Occasionally, we refer to subsets $B_i\subseteq A$ as the \emph{subreduct
  universes} of $\bA$, we call $\bB$ the
\emph{base algebra} of $\bA$, and $T$ the sequence of \emph{tie-points} of $\bA$.

Before proving some general results about the basic structure of the
congruence lattice of an overalgebra, we present the first example,
discovered by Ralph Freese, of a finite algebra with congruence lattice
isomorphic to the second lattice in  Figure~\ref{fig:sevens}.
(All computational experiments described in this paper rely on
two open source programs, GAP~\cite{GAP4} and the Universal Algebra
Calculator~\cite{uacalc}. Source code for reproducing our results is
available at \url{http://williamdemeo.wordpress.com/software/overalgebras/}.)
\begin{example}
\label{ex:3.1}
Consider a finite permutational algebra $\bB = \<B, F\>$
with congruence lattice $\Con\bB \cong M_4$ (Figure~\ref{fig:ConS3}).
The right regular $S_3$-set---%
i.e., the group $S_3$
acting on itself by right multiplication---is one such algebra.  In
\GAP,
\begin{small}
\begin{verbatim}
gap> g:=SymmetricGroup(3); # The symmetric group on 3 letters.
gap> B:=[(),(1,2,3),(1,3,2),(1,2),(1,3),(2,3)]; # The elements of g.
gap> G:=Action(g,B,OnRight); # Group([(1,2,3)(4,5,6), (1,4)(2,6)(3,5)])
\end{verbatim}
\end{small}

In our computational examples, we prefer to use 0-offset notation since this is
the convention used by the \uacalc\ software.  Thus, in the present example we define
the universe of the $S_3$-set described above to be
$B = \{0, 1,\dots, 5\}$, instead of $\{1, 2, \dots, 6\}$.
As such, the partitions of $B$ corresponding to nontrivial congruence relations
of the algebra are given as follows:

\begin{small}
\begin{verbatim}
gap> for b in AllBlocks(G) do Print(Orbit(G,b,OnSets)-1,"\n"); od;
[ [ 0, 1, 2 ], [ 3, 4, 5 ] ]
[ [ 0, 3 ], [ 1, 4 ], [ 2, 5 ] ]
[ [ 0, 4 ], [ 1, 5 ], [ 2, 3 ] ]
[ [ 0, 5 ], [ 1, 3 ], [ 2, 4 ] ]
\end{verbatim}
\end{small}

\noindent 
Next, we create an algebra
in \uacalc\ format using the two generators of the
 group as basic operations.  This can be accomplished using
our \GAP\ script {\tt gap2uacalc.g} as follows:

\begin{small}
\begin{verbatim}
gap> Read("gap2uacalc.g");
gap> gset2uacalc([G,"S3action"]);
\end{verbatim}
\end{small}

\noindent 
This creates a \uacalc\ file 
specifying an algebra which has universe $B = \{0, 1, \dots, 5\}$ and two
basic unary operations
$g_0 = (1\;  2 \; 0\;  4\;  5\;  3)$
and $g_1 = (3\;  5\; 4\; 0\;  2\;  1)$.
These operations are the permutations $(0,1,2)(3,4,5)$ and $(0,3)(1,5)(2,4)$,
which are (0-offset versions of) the generators
of the $S_3$-set appearing in the \GAP\ output above.
Figure~\ref{fig:ConS3} displays the congruence lattice of this algebra.

\begin{figure}[h!]
  \centering
    \begin{tikzpicture}[scale=.9]
      \node (250) at (2.5,0)  [draw, circle, inner sep=\dotsize] {};
      \node (253) at (2.5,3)  [draw, circle, inner sep=\dotsize] {};
      \foreach \i in {1,2, 3, 4}
      {
        \node (\i15) at (\i,1.5)  [draw, circle, inner sep=\dotsize] {};
        \draw[semithick] (250) to (\i15) to (253);
      }
      \draw (115) node [left] {$\alpha$};
      \draw (215) node [left] {$\beta$};
      \draw (315) node [right] {$\gamma$};
      \draw (415) node [right] {$\delta$};
      \draw (2.8,3.15) node {$1_{B}$};
      \draw (2.83,-.15) node {$0_{B}$};
    \end{tikzpicture}
  \caption{Congruence lattice of the right regular $S_3$-set, where congruences
    $\alpha$, $\beta$, $\gamma$,
    and $\delta$ correspond to the partitions
        $| 0, 1, 2 | 3, 4, 5|$, $| 0, 3 | 1, 4 | 2, 5 |$, $| 0, 4| 1, 5| 2, 3|$,
        and
    $| 0, 5| 1, 3| 2, 4|$, respectively.
  }
  \label{fig:ConS3}
\end{figure}
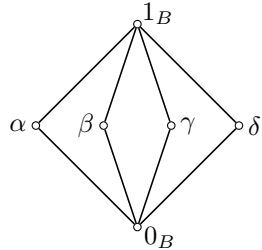

We now construct an overalgebra that ``doubles'' the congruence
$\alpha$ corresponding to the partition $| 0, 1, 2 | 3, 4, 5|$. 
(From now on, we will identify a congruence relation with the corresponding
partition of the underlying set, and write, for example, $\alpha = | 0, 1, 2 |
3, 4, 5|$.) 
Note that $\alpha$ can be generated by the pair $(0,2)$---that is, $\alpha =
\Cg^\bB(0,2)$---and
the desired doubling of $\alpha$ can be achieved by
choosing the tie-point sequence $t_1, t_2 = 0, 2$ in the overalgebra
construction.
(Theorem~\ref{thm1} below will reveal why this choice works.  Note that, 
in this simple example, no non-trivial partition of the
  tie-point indices is required.)

Our \GAP\ function {\tt Overalgebra} carries out the construction, and is
invoked as follows:

\begin{small}
\begin{verbatim}
gap> Read("Overalgebras.g");
gap> Overalgebra([G, [0,2]]);
\end{verbatim}
\end{small}

\noindent 
This results in an
algebra with universe $A = B \cup B_1 \cup B_2 = \{ 0, 1, 2, 3, 4, 5\} \cup
\{0, 6, 7, 8, 9, 10\} \cup\{ 11, 12, 2, 13, 14, 15\}$,
and the following operations:

\smallskip 
\begin{center}
\begin{tabular}{c|r|r|r|r|r|r|r|r|r|r|r|r|r|r|r|r}
 &0&1&2&3&4&5&6&7&8&9&10&11&12&13&14&15\\
\hline
$e_0$ & 0& 1& 2& 3& 4& 5& 1& 2& 3& 4& 5& 0& 1& 3& 4& 5\\
$e_1$ & 0& 6& 7& 8& 9& 10& 6& 7& 8& 9& 10& 0& 6& 8& 9& 10\\
$e_2$ &11& 12& 2& 13& 14& 15& 12& 2& 13& 14& 15& 11& 12& 13& 14& 15\\
$s_1$  & 0& 1& 2& 3& 4& 5& 0& 0& 0& 0& 0& 2& 2& 2& 2& 2\\
$g_0 e_0$&1 &2 &0 &4 &5 &3 &2 &0 &4 &5 &3 &1 &2 &4 &5 &3\\
$g_1 e_0$& 3 &5 &4 &0 &2 &1 &5 &4 &0 &2 &1 &3 &5 &0 &2 &1
\end{tabular}
\end{center}
\smallskip 
\noindent 
If 
$F_A=\{e_0, e_1, e_2, s_1, g_0 e_0, g_1 e_0\}$, then the
algebra $\<A, F_A\>$ has the congruence lattice shown in Figure~\ref{fig:OverAlgebra-S3-0-2}.
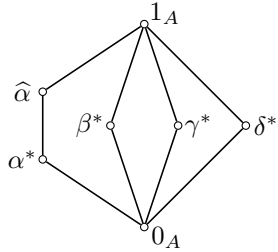
\begin{figure}[h!]
  \centering
    \begin{tikzpicture}[scale=.9]
      \node (70) at (6.5,0)  [draw, circle, inner sep=\dotsize] {};
      \node (71) at (7,1.5)  [draw, circle, inner sep=\dotsize] {};
      \node (73) at (6.5,3)  [draw, circle, inner sep=\dotsize] {};
      \node (61) at (6,1.5)  [draw, circle, inner sep=\dotsize] {};
      \node (51) at (5,1)  [draw, circle, inner sep=\dotsize] {};
      \node (52) at (5,2)  [draw, circle, inner sep=\dotsize] {};
      \node (81) at (8,1.5)  [draw, circle, inner sep=\dotsize] {};
      \draw[semithick]
      (70) to (51) to (52) to (73)
      (70) to (61) to (73)
      (70) to (71) to (73)
      (70) to (81) to (73);
      \draw (4.7,1) node {$\alpha^*$};
      \draw (4.7,2) node {$\widehat{\alpha}$};
      \draw (5.7,1.5) node {$\beta^*$};
      \draw (7.3,1.5) node {$\gamma^*$};
      \draw (8.3,1.5) node {$\delta^*$};
      \draw (6.8,3.15) node {$1_A$};
      \draw (6.83,-.15) node {$0_A$};
    \end{tikzpicture}
  \caption{Congruence lattice of the overalgebra of the $S_3$-set with
    intersection points 0 and 2.}
  \label{fig:OverAlgebra-S3-0-2}
\end{figure}

The congruence relations in Figure~\ref{fig:OverAlgebra-S3-0-2} are as
follows:
 \begin{align*}
 \widehat{\alpha} &=|0,1,2,6,7,11,12|3,4,5|8,9,10,13,14,15| \\
 \alpha^* &=|0,1,2,6,7,11,12|3,4,5|8,9,10|13,14,15| \\
 \beta^*&=|0,3,8|1,4|2,5,15|6,9|7,10|11,13|12,14| \\
 \gamma^*&=|0,4,9|1,5|2,3,13|6,10|7,8|11,14|12,15| \\
 \delta^*&=|0,5,10|1,3|2,4,14|6,8|7,9,11,15|12,13|.
 \end{align*}

It is important to note that the resulting congruence lattice depends
on our choice of which congruence to ``expand,'' which is controlled by
our specification of the tie-points of the overalgebra.
For example, suppose we want one of the congruences having three
blocks, say, $\beta = \Cg^\bB(0,3) =| 0, 3 | 1, 4 | 2, 5 |$, to have a
non-trivial $\resB$-inverse image $[\beta^*, \widehat{\beta}]$.  Then we could
specify the sequence $t_1, t_2$ to be $0, 3$ by invoking
the command
{\tt Overalgebra([G, [0,3]])}. (In this instance we arrive at the same congruence
  lattice if we let $t_1, t_2$ be $2, 5$ or $1, 4$.)
This produces an overalgebra with universe
\[
A = B \cup B_1 \cup B_2
 = \{0, 1,  2,  3,  4,  5\} \cup \{ 0, 6,  7,  8,  9, 10\} \cup
\{11, 12, 13, 3, 14, 15\}
\]
and congruence lattice shown in Figure~\ref{fig:OverAlgebra-S3-0-3},
where
 \begin{align*}
 \alpha^* &=|0,1,2,6,7|3,4,5,14,15|8,9,10|11,12,13| \\
 \widehat{\beta} &=|0,3,8,11|1,4|2,5|6,9,12,14|7,10,13,15| \\
 \beta_{\eps}&=|0,3,8,11|1,4|2,5|6,9,12,14|7,10|13,15| \\
 \beta_{\eps'}&=|0,3,8,11|1,4|2,5|6,9|7,10,13,15|12,14| \\
 \beta^*&=|0,3,8,11|1,4|2,5|6,9|7,10|12,14|13,15| \\
 \gamma^*&=|0,4,9|1,5|2,3,13|6,10|7,8|11,14|12,15| \\
 \delta^*&=|0,5,10|1,3,12|2,4|6,8|7,9|11,15|13,14|.
 \end{align*}

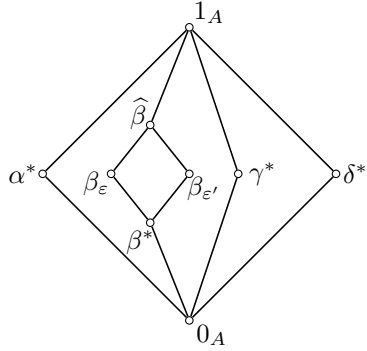
\begin{figure}[h!]
  \centering
    \begin{tikzpicture}[scale=1.3]
      \node (70) at (6.5,0)  [draw, circle, inner sep=\dotsize] {};
      \node (71) at (7,1.5)  [draw, circle, inner sep=\dotsize] {};
      \node (73) at (6.5,3)  [draw, circle, inner sep=\dotsize] {};
      \node (61) at (6.1,1)  [draw, circle, inner sep=\dotsize] {};
      \node (62) at (6.1,2)  [draw, circle, inner sep=\dotsize] {};
      \node (63) at (5.7,1.5)  [draw, circle, inner sep=\dotsize] {};
      \node (64) at (6.5,1.5)  [draw, circle, inner sep=\dotsize] {};
      \node (51) at (5,1.5)  [draw, circle, inner sep=\dotsize] {};
      \node (81) at (8,1.5)  [draw, circle, inner sep=\dotsize] {};
      \draw[semithick]
      (70) to (51) to (73)
      (70) to (61) to (63) to (62) to (73)
      (61) to (64) to (62)
      (70) to (71) to (73)
      (70) to (81) to (73);
      \draw (4.8,1.5) node {$\alpha^*$};
      \draw (5.98,2.12) node {$\widehat{\beta}$};
      \draw (5.55,1.4) node {$\beta_\eps$};
      \draw (6.65,1.35) node {$\beta_{\eps'}$};
      \draw (6,.8) node {$\beta^*$};
      \draw (7.25,1.5) node {$\gamma^*$};
      \draw (8.2,1.5) node {$\delta^*$};
      \draw (6.7,3.15) node {$1_A$};
      \draw (6.73,-.15) node {$0_A$};
    \end{tikzpicture}
  \caption{Congruence lattice of the overalgebra of the $S_3$-set with
    intersection points 0 and 3.}
  \label{fig:OverAlgebra-S3-0-3}
\end{figure}

\end{example}

Before proceeding, we require some more notation.  Given a
set $X$, we let $\Eq(X)$ denote the set of equivalence relations on the set $X$, and
we let $\one, \two$, and $\three$ denote the abstract one, two,
and three element lattices, respectively. For example, $\Eq(1) \cong
\one$ and $\Eq(2) \cong \two$.  Also, let us agree that
$\Eq(0) \cong \one$.

Below we present a theorem that describes the basic structure of the congruence
lattice of an overalgebra constructed as described in this
section.  In particular, the theorem explains why the interval
$[\alpha^*, \widehat{\alpha}]\cong\two $ appears in the first example above,
while $[\beta^*, \widehat{\beta}]\cong\two\times \two $ appears
in the second.

If $\beta \in \Con \bB$ and if we
denote by $C_1, \dots, C_m$ the congruence classes of $\beta$, then
for each $1\leq i\leq K$ the isomorphism $\pi_i$ yields a corresponding
congruence relation $\beta^i \in \Con \bB_i$, which partitions the set
$B_i$ into congruence classes $\pi_i(C_1), \dots, \pi_i(C_m)$.  We
denote the $r$-th class of $\beta^i$ by $C_r^i=\pi_i(C_r)$.

If $t_1, t_2, \dots, t_K$ is the sequence of tie-points on which an
overalgebra is based,
we let $\sI_r = \{i \mid t_i \in C_r\}$ denote the indices of those
tie-points that lie in the $r$-th congruence class of $\beta$.
Note that if $i\in \sI_r$, then the $r$-th congruence class of
$\beta^i$ intersects the $r$-th congruence class of $\beta$ at the tie-point
$t_i$, and we have $\{t_i\} = C_r^i\cap C_r$, 
and $t_i/\beta^i = C_r^i$.  Finally, 
we let $\beta^{0}$ denote $\beta$.
\begin{theorem}
\label{thm1}
Let $\bB$ be a finite unary algebra and let $\bA$ be the overalgebra of $\bB$ constructed,
as described above, from the sequence $t_1, t_2, \dots, t_K$ and the partition $|\sT_1|\sT_2|
\cdots|\sT_N|$ of the indices $\{1, 2, \dots, K\}$.
Given $\beta \in \Con \bB$ with congruence classes $C_1, \dots, C_m$, define the relations
  \begin{equation}
    \label{eq:star}
\beta^\star= \bigcup_{k=0}^K \beta^{k} \cup \bigcup_{r=1}^m
\bigl(C_r \cup \bigcup_{i\in \sI_r} C_r^i \bigr)^2,
  \end{equation}
and
  \begin{equation}
    \label{eq:betahat}
    \tbeta =
    \beta^\star \cup
    \bigcup_{n=1}^N
    \bigcup_{r=1}^m
    \bigcup^m_{\begin{subarray}{l}\ell=1 \\\ell \neq r \end{subarray}}
    \bigl(\bigcup_{i \in \sT_n \cap \sI_r} C^i_{\ell}\bigr)^2.
  \end{equation}
Then,
\begin{enumerate}[\rm(i)] 
\item
$\beta^\star=\beta^*$, the minimal $\theta\in \Con\bA$ such that $\theta\resB = \beta$;
\item $\tbeta=\widehat{\beta}$, the maximal $\theta \in \Con\bA$ such that $\theta\resB = \beta$;
\item
\label{item-iii}
the interval $[\beta^\star, \tbeta]$ in $\Con\bA$
  contains every equivalence relation on $A$ between $\beta^\star$ and $\tbeta$,
and satisfies
  \begin{equation}
    \label{eq:item-iii}
          [\beta^\star, \tbeta]
          =
          \{\theta \in \Eq(A) \suchthat \beta^\star \subseteq \theta \subseteq \tbeta \}
          \cong \prod_{r=1}^m \prod_{n=1}^N (\Eq |\sT_n \cap \sI_r|)^{m-1}.
  \end{equation}
\end{enumerate}
\end{theorem}

\begin{remarks}
If a block $C_r$ of $\beta$ contains no tie-point, then
$\sI_r = \emptyset$. In such cases,
for all $n$,
$\Eq |\sT_n \cap \sI_r| = \Eq |\sI_r| = \Eq (0)
\cong \one$.
  Similarly, if the block $C_r$ contains a single tie-point, then $|\sT_n \cap
  \sI_r|\leq 1$ for all $n$, and
  again the corresponding terms in~(\ref{eq:item-iii}) contribute nothing to the
  product.
In fact, the term of index $(n,r)$ in~(\ref{eq:item-iii}) is nontrivial if and only if there
are at least two tie-points (say, $t_i, t_j$) in block $C_r$ with indices (say, $i, j$) in block $\sT_n$.

Next we remark that the blocks of the relation $\beta^\star$ are, for $1\leq r\leq m$,
\begin{equation}
  \label{eq:betastartclasses}
C_r \cup \bigcup_{i\in \sI_r}C_r^i \quad
\text{and} \quad
C^j_r \quad (1\leq j\leq K, \ j\notin \sI_r).
\end{equation}
The blocks of $\tbeta$ are, for $1\leq r\leq m$ and
$1\leq n \leq N$,
\begin{equation}
  \label{eq:hbetaclasses}
C_r \cup \bigcup_{i\in \sI_r}C_r^i   \quad \text{and} \quad
\bigcup_{j\in \sT_n \cap \sI_r}C^j_s \quad (1\leq s\leq m, \ s\neq r).
\end{equation}
In the special case $N=1$ (the trivial partition), (\ref{eq:hbetaclasses}) is simply
\[
C_r \cup \bigcup_{i\in \sI_r}C_r^i   \quad \text{and} \quad
\bigcup_{j\in \sI_r}C^j_s \quad (1\leq s\leq m, \ s\neq r).
\]
Figures~\ref{fig:overalgebra} and~\ref{fig:overalgebra1} depict the blocks of $\beta^\star$
and $\tbeta$ for a small example illustrating the special case $N=1$.
These diagrams serve as a rough guide to intuition,
and make the proof of Theorem~\ref{thm1} easier to follow.
\end{remarks}

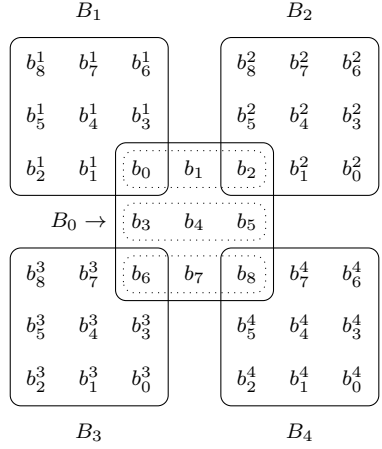
\begin{figure}
  \centering
      {\scalefont{.8}
        \begin{tikzpicture}[scale=.7]
          \draw[rounded corners] (-1.5,-1.5) rectangle (1.5,1.5);
          \draw[rounded corners] (.5,.5) rectangle (3.5,3.5);
          \draw[rounded corners] (.5,-3.5) rectangle (3.5,-.5);
          \draw[rounded corners] (-3.5,-3.5) rectangle (-.5,-.5);
          \draw[rounded corners] (-3.5,.5) rectangle (-.5,3.5);
          \draw[rounded corners, dotted] (-1.35,.65) rectangle (1.35,1.35);
          \draw[rounded corners, dotted] (-1.35,-.35) rectangle (1.35,.35);
          \draw[rounded corners, dotted] (-1.35,-1.35) rectangle (1.35,-.65);

          \draw (-3,1) node {$b^1_2$};
          \draw (-2,1) node {$b^1_1$};
          \draw (-1,1) node {$b_0$};
          \draw (-3,3) node {$b^1_8$};
          \draw (-2,3) node {$b^1_7$};
          \draw (-1,3) node {$b^1_6$};
          \draw (-3,2) node {$b^1_5$};
          \draw (-2,2) node {$b^1_4$};
          \draw (-1,2) node {$b^1_3$};

          \draw ( 0,1) node {$b_1$};

          \draw (3,1) node {$b^2_0$};
          \draw (2,1) node {$b^2_1$};
          \draw (1,1) node {$b_2$};
          \draw (3,3) node {$b^2_6$};
          \draw (2,3) node {$b^2_7$};
          \draw (1,3) node {$b^2_8$};
          \draw (3,2) node {$b^2_3$};
          \draw (2,2) node {$b^2_4$};
          \draw (1,2) node {$b^2_5$};

          \draw (-1,0) node {$b_3$};
          \draw ( 0,0) node {$b_4$};
          \draw ( 1,0) node {$b_5$};

          \draw (-3,-1) node {$b^3_8$};
          \draw (-2,-1) node {$b^3_7$};
          \draw (-1,-1) node {$b_6$};
          \draw (-3,-2) node {$b^3_5$};
          \draw (-2,-2) node {$b^3_4$};
          \draw (-1,-2) node {$b^3_3$};
          \draw (-3,-3) node {$b^3_2$};
          \draw (-2,-3) node {$b^3_1$};
          \draw (-1,-3) node {$b^3_0$};

          \draw (3,-1) node {$b^4_6$};
          \draw (2,-1) node {$b^4_7$};
          \draw (1,-1) node {$b_8$};
          \draw (3,-3) node {$b^4_0$};
          \draw (2,-3) node {$b^4_1$};
          \draw (1,-3) node {$b^4_2$};
          \draw (3,-2) node {$b^4_3$};
          \draw (2,-2) node {$b^4_4$};
          \draw (1,-2) node {$b^4_5$};

          \draw (0,-1) node {$b_7$};

          \draw (-2.2,0) node {$B_0 \rightarrow $};
          \draw (-2, 4) node {$B_1$};
          \draw ( 2, 4) node {$B_2$};
          \draw (-2,-4) node {$B_3$};
          \draw ( 2,-4) node {$B_4$};
        \end{tikzpicture}
      }
      \caption{The universe $A = B_0 \cup \cdots \cup B_4$ of an
       example overalgebra; dotted lines surround each congruence class $C_i$ of
       $\beta$; the tie-points are $t_1, t_2, t_3, t_4 = b_0, b_2, b_6, b_8$;
       if 
       $\sI_r = \{i \mid t_i \in C_r\}$, then
      $\sI_1 = \{1, 2\}$, $\sI_2 = \emptyset$, and $\sI_3 = \{3, 4\}$.
}
      \label{fig:overalgebra}
\end{figure}

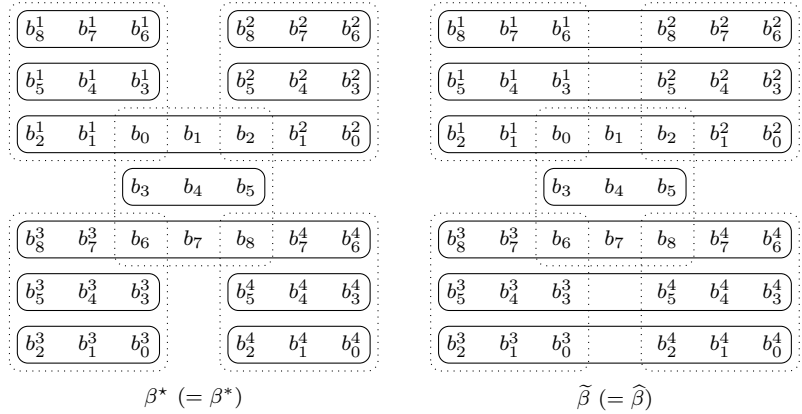
\begin{figure}[h]
  \centering
      {\scalefont{.8}
        \begin{tikzpicture}[scale=.7]
          \draw[rounded corners,dotted] (-1.5,-1.5) rectangle (1.5,1.5);
          \draw[rounded corners,dotted] (.5,.5) rectangle (3.5,3.5);
          \draw[rounded corners,dotted] (.5,-3.5) rectangle (3.5,-.5);
          \draw[rounded corners,dotted] (-3.5,-3.5) rectangle (-.5,-.5);
          \draw[rounded corners,dotted] (-3.5,.5) rectangle (-.5,3.5);

          \draw[rounded corners, dotted] (6.5,-1.5) rectangle (9.5,1.5);
          \draw[rounded corners, dotted] (8.5,.5) rectangle (11.5,3.5);
          \draw[rounded corners, dotted] (8.5,-3.5) rectangle (11.5,-.5);
          \draw[rounded corners, dotted] (4.5,-3.5) rectangle (7.5,-.5);
          \draw[rounded corners, dotted] (4.5,.5) rectangle (7.5,3.5);

          \draw (0, -4) node {$\beta^\star$ $(=\beta^*)$};
          \draw[rounded corners] (-1.35,-.35) rectangle (1.35,.35);
          \draw[rounded corners] (-3.35,.65) rectangle (3.35,1.35);
          \draw[rounded corners] (-3.35,-.65) rectangle (3.35,-1.35);
          \draw[rounded corners] (-3.35,1.65) rectangle (-.65,2.35);
          \draw[rounded corners] (-3.35,2.65) rectangle (-.65,3.35);
          \draw[rounded corners] (-3.35,-1.65) rectangle (-.65,-2.35);
          \draw[rounded corners] (-3.35,-2.65) rectangle (-.65,-3.35);
          \draw[rounded corners] (4-3.35,1.65) rectangle (4-.65,2.35);
          \draw[rounded corners] (4-3.35,2.65) rectangle (4-.65,3.35);
          \draw[rounded corners] (4-3.35,-1.65) rectangle (4-.65,-2.35);
          \draw[rounded corners] (4-3.35,-2.65) rectangle (4-.65,-3.35);

          \draw (8, -4) node {$\tbeta$ $(=\hbeta)$};
          \draw[rounded corners] (8-1.35,-.35) rectangle (8+1.35,.35); 
          \draw[rounded corners] (8-3.35,.65) rectangle (11.35,1.35); 
          \draw[rounded corners] (8-3.35,-.65) rectangle (11.35,-1.35);
          \draw[rounded corners] (8-3.35,1.65) rectangle (4+8-.65,2.35);
          \draw[rounded corners] (8-3.35,2.65) rectangle (4+8-.65,3.35);
          \draw[rounded corners] (8-3.35,-1.65) rectangle (4+8-.65,-2.35);
          \draw[rounded corners] (8-3.35,-2.65) rectangle (4+8-.65,-3.35);

          \draw (-3,1) node {$b^1_2$};
          \draw (-2,1) node {$b^1_1$};
          \draw (-1,1) node {$b_0$};
          \draw (-3,3) node {$b^1_8$};
          \draw (-2,3) node {$b^1_7$};
          \draw (-1,3) node {$b^1_6$};
          \draw (-3,2) node {$b^1_5$};
          \draw (-2,2) node {$b^1_4$};
          \draw (-1,2) node {$b^1_3$};

          \draw ( 0,1) node {$b_1$};

          \draw (3,1) node {$b^2_0$};
          \draw (2,1) node {$b^2_1$};
          \draw (1,1) node {$b_2$};
          \draw (3,3) node {$b^2_6$};
          \draw (2,3) node {$b^2_7$};
          \draw (1,3) node {$b^2_8$};
          \draw (3,2) node {$b^2_3$};
          \draw (2,2) node {$b^2_4$};
          \draw (1,2) node {$b^2_5$};

          \draw (-1,0) node {$b_3$};
          \draw ( 0,0) node {$b_4$};
          \draw ( 1,0) node {$b_5$};

          \draw (-3,-1) node {$b^3_8$};
          \draw (-2,-1) node {$b^3_7$};
          \draw (-1,-1) node {$b_6$};
          \draw (-3,-2) node {$b^3_5$};
          \draw (-2,-2) node {$b^3_4$};
          \draw (-1,-2) node {$b^3_3$};
          \draw (-3,-3) node {$b^3_2$};
          \draw (-2,-3) node {$b^3_1$};
          \draw (-1,-3) node {$b^3_0$};

          \draw (3,-1) node {$b^4_6$};
          \draw (2,-1) node {$b^4_7$};
          \draw (1,-1) node {$b_8$};
          \draw (3,-3) node {$b^4_0$};
          \draw (2,-3) node {$b^4_1$};
          \draw (1,-3) node {$b^4_2$};
          \draw (3,-2) node {$b^4_3$};
          \draw (2,-2) node {$b^4_4$};
          \draw (1,-2) node {$b^4_5$};

          \draw (0,-1) node {$b_7$};


          \draw (8-3,1) node {$b^1_2$};
          \draw (8-2,1) node {$b^1_1$};
          \draw (8-1,1) node {$b_0$};
          \draw (8-3,3) node {$b^1_8$};
          \draw (8-2,3) node {$b^1_7$};
          \draw (8-1,3) node {$b^1_6$};
          \draw (8-3,2) node {$b^1_5$};
          \draw (8-2,2) node {$b^1_4$};
          \draw (8-1,2) node {$b^1_3$};

          \draw (8,1) node {$b_1$};

          \draw (8+3,1) node {$b^2_0$};
          \draw (8+2,1) node {$b^2_1$};
          \draw (8+1,1) node {$b_2$};
          \draw (8+3,3) node {$b^2_6$};
          \draw (8+2,3) node {$b^2_7$};
          \draw (8+1,3) node {$b^2_8$};
          \draw (8+3,2) node {$b^2_3$};
          \draw (8+2,2) node {$b^2_4$};
          \draw (8+1,2) node {$b^2_5$};

          \draw (8-1,0) node {$b_3$};
          \draw (8+ 0,0) node {$b_4$};
          \draw (8+ 1,0) node {$b_5$};

          \draw (8-3,-1) node {$b^3_8$};
          \draw (8-2,-1) node {$b^3_7$};
          \draw (8-1,-1) node {$b_6$};
          \draw (8-3,-2) node {$b^3_5$};
          \draw (8-2,-2) node {$b^3_4$};
          \draw (8-1,-2) node {$b^3_3$};
          \draw (8-3,-3) node {$b^3_2$};
          \draw (8-2,-3) node {$b^3_1$};
          \draw (8-1,-3) node {$b^3_0$};

          \draw (8+3,-1) node {$b^4_6$};
          \draw (8+2,-1) node {$b^4_7$};
          \draw (8+1,-1) node {$b_8$};
          \draw (8+3,-2) node {$b^4_3$};
          \draw (8+2,-2) node {$b^4_4$};
          \draw (8+1,-2) node {$b^4_5$};
          \draw (8+3,-3) node {$b^4_0$};
          \draw (8+2,-3) node {$b^4_1$};
          \draw (8+1,-3) node {$b^4_2$};

          \draw (8+0,-1) node {$b_7$};

        \end{tikzpicture}
      }
      \caption{Solid lines delineate the congruence classes of $\beta^\star$ (left) and
        $\tbeta$ (right); dotted lines surround the sets $B_i$.}
      \label{fig:overalgebra1}
\end{figure}

\begin{proof}[Proof of Theorem~\ref{thm1}]
(i) 
We first check that $\beta^\star\in \Con\bA$.
 It is
easy to see that $\beta^\star$ is an equivalence relation on $A$.  To see that it
is a congruence relation of $\bA$, we prove $f(\beta^\star) \subseteq \beta^\star$ for
all $f\in F_A$, where
\[
F_A \defeq  \{g e_0 \suchthat g\in F_B\} \cup \{e_k \suchthat 0\leq k \leq K\} \cup \{s_n \suchthat 1 \leq n
\leq N\}.
\]
That is, we prove: if
$(x,y)\in \beta^\star$ and $f\in F_A$, then $(f(x), f(y))\in \beta^\star$.\\[6pt]
Fix $0\leq k \leq K$ and consider the action of the map $e_k$ on the congruence
classes of $\beta^\star$ described in~(\ref{eq:betastartclasses}).  For all $1\leq r \leq
m$, we have $e_k(C_r^j) = C_r^k$, for all $0\leq j\leq K$, and so
$e_k(C_r \cup \bigcup_{i\in \sI_r} C_r^i) = C_r^k$.  It follows that for all $0\leq
k\leq K$ the map $e_k$ takes blocks of $\beta^\star$ into blocks of
$\beta^\star$, thus, $e_k(\beta^\star) \subseteq \beta^\star$.  In
particular, $e_0$ takes blocks of $\beta^\star$ into blocks of $\beta$, so
$g e_0(\beta^\star) \subseteq \beta^\star$ for all $g\in F_B$.  Next note that
for $1\leq n\leq N$, $0\leq r \leq K$, $0\leq j\leq K$ the
map $s_n$ acts as the identity on $C_r^j$ when $j\notin \sT_n$, otherwise it maps
$C_r^j$  to the point $t_j$. It follows that $s_n
(C_r \cup \bigcup_{i\in \sI_r}C_r^i)\subseteq C_r \cup \bigcup_{i\in
  \sI_r}C_r^i$, since the union is over $\sI_r = \{i \mid t_i \in C_r\}$.
Therefore, $s_n(\beta^\star) \subseteq \beta^\star$ for all $1\leq n \leq N$,
and we have proved
$\beta^\star\in \Con\bA$.

Since $\beta = \beta_0 \subseteq \beta^\star$, we have
$\beta^\star \resB = \beta$.
Therefore, $\beta^\star \geq \beta^*$, by the
residuation lemma of Section~\ref{sec:residuation-lemma}.
To complete the proof of (i), we show that
$\beta \subseteq \eta \in \Con \bA$ implies
$\beta^\star\leq \eta$.
If $\beta \subseteq \eta\in \Con \bA$, then $\bigcup \beta^{k} \subseteq
\eta$, since
for each $0\leq k\leq K$ and for each pair
$(u,v)\in \beta^{k}$ there exists a pair $(x,y)\in \beta$ with
$(e_k(x), e_k(y)) = (u,v)$, and $(x,y)\in \beta\subseteq \eta$ implies
$(u,v) = (e_k(x), e_k(y)) \in \eta$.
The second term of~(\ref{eq:star}) belongs to
$\eta$ by transitivity.  Indeed, suppose $(x,y)$ is an arbitrary element of that
term, with, say, $(x, t_i) \in \beta^{i}$, $(y, t_j)\in \beta^{j}$, and $(t_i, t_j) \in
\beta$.  As we just observed, 
$\beta$, $\beta^{i}$,
and $\beta^{j}$ are subsets of $\eta$,
so 
$x \mathrel{\beta^{i}} t_i \mathrel{\beta} t_j \mathrel{\beta^{j}} y$ implies
$(x,y)\in \eta$.

\smallskip 
(ii) 
Clearly $\tbeta$ is an equivalence on $A$.  To see that it
is a congruence relation of $\bA$, we prove $f(\tbeta) \subseteq \tbeta$ for
all $f\in F_A$.
Fix $(x,y)\in \tbeta$.  If
$(x,y)\in \beta^\star$, then $(f(x), f(y)) \in \beta^\star$ holds for all $f\in F_A$, as
in part (i).
If $(x,y)\notin \beta^\star$, then
$x \in C^j_{\ell}$ and
$y\in C^k_{\ell}$
for some $j, k \in \sT_n \cap \sI_r$,  with
$1\leq n \leq N$,
$1\leq r \leq m$, and $\ell\neq r$.
In this case, $x$ and $y$ are in the $\ell$-th blocks of $B_j$ and $B_k$, respectively, so
for each $0\leq i \leq K$
both $e_i(x)$ and $e_i(y)$ belong to $C_\ell^i$, so
$(e_i(x), e_i(y)) \in \beta^{i}$.
In particular, $(e_0(x), e_0(y)) \in \beta$, so
$(g e_0(x), g e_0(y)) \in \beta$ for all $g\in F_B$.
Also, $(s_n(x),s_n(y))=(t_j,t_k) \in C_r\times C_r \subseteq \beta$, while
$(s_{n'}(x),s_{n'}(y))=(x,y)\in \tbeta$, when $n'\neq n$.
This proves that $(f(x), f(y)) \in \tbeta$ for all $f\in F_A$.
Whence $\tbeta \in \Con\bA$.
(Note that we have proved: $(x,y)\in \tbeta$ implies $(f(x), f(y))\in \beta^\star$ for all $f\in F_A$,
except when $f=s_n$ acts as the identity on both $x$ and $y$.
This will be useful in the proof of part (\ref{item-iii}) below.)

Notice that $\tbeta\resB = \beta$.  Therefore, by the residuation lemma
of Section~\ref{sec:residuation-lemma}, we have $\tbeta \leq \widehat{\beta}$.
To prove $\tbeta=\widehat{\beta}$, we suppose $(x,y)\notin \tbeta$ and show
$(x,y) \notin \widehat{\beta}$.
Recall that the map $\hatmap\colon  \Con\bB \rightarrow \Con\bA$ is given by
\[
\widehat{\beta} = \{(x,y) \in A^2 \suchthat \text{ for all }
f\in \Pol_1(\bA), \, (e_0f(x), e_0f(y))\in \beta \}.
\]
Let
$x \in C_p^j$ and
$y\in C^k_q$, for some $1\leq p, q
\leq m$ and $1\leq j, k \leq K$.  If $p\neq q$, then
$e_0(x) \in C_p$ and $e_0(y)\in C_q$---distinct $\beta$ classes---so
$(x,y) \notin \widehat{\beta}$.
Suppose $p=q$.  If both $j$ and $k$ belong to $\sT_n\cap \sI_r$ for some
$1\leq n \leq N$, $1\leq r\leq m$, then $(x,y)\in \tbeta$, contradicting our
assumption.
If $(j,k) \in \sT_n^2$ for some $n$, then for all $r$ we have $(j,k) \notin \sI_r^2$, so
$(e_0s_n(x), e_0s_n(y)) = (t_j, t_k) \notin \beta$ and
 $(x,y) \notin \widehat{\beta}$.
If $(j,k) \in \sI_r^2$ for some $r$, then for all $n$ we have $(j,k) \notin
\sT_n^2$.  Without loss of generality, assume $j \in \sT_n$. Then
$(e_0s_n(x), e_0s_n(y)) = (t_j, e_0(y))$, which does not belong to $\beta$.
For, $t_j\in C_r$ while $e_0(y) \in C_q$, and $q\neq r$ (otherwise $(x,y)\in \beta^\star \leq \tbeta$).
Thus, $(x,y) \notin \widehat{\beta}$.  Finally, suppose that for all $n$
and $r$ we have $(j,k) \notin \sT_n^2$ and $(j,k) \notin \sI_r^2$.
Let $j\in \sT_n$ and $k\in \sT_{n'}$. Then
$(e_0 s_n s_{n'}(x), e_0s_n s_{n'} (y)) = (t_j, t_k)\notin \beta$,
so $(x,y) \notin \widehat{\beta}$.

\smallskip 
(iii)
Note that every equivalence relation $\theta$ on $A$ with
$\beta^\star \subseteq \theta \subseteq \tbeta$ satisfies
$f(\theta)\subseteq \theta$ for all $f\in F_A$, and is therefore a congruence
relation of $\bA$. Indeed, in proving that $\tbeta$ is a congruence of $\bA$,
we saw that $f(\tbeta)\subseteq \beta^\star$ for all $f\in F_A$, except when $f
= s_n$, in which case $f$ acts as the identity on some blocks of $\tbeta$ and
maps other blocks of $\tbeta$ to tie-points.  It follows that
$f(\theta)\subseteq \theta$ for all equivalence relations
$\beta^\star \subseteq \theta \subseteq \tbeta$.
Therefore, the interval $[\beta^\star, \tbeta]$ in $\Con\bA$ is
$\{\theta \in \Eq(A) \suchthat \beta^\star \subseteq \theta \subseteq \tbeta \}$.
To complete the proof, we must show that this interval is isomorphic to the lattice
$\prod_{r=1}^m \prod_{n=1}^N (\Eq |\sT_n \cap \sI_r|)^{m-1}$.
This follows from a standard fact about intervals $[\zeta, \eta]$ for $\zeta\leq
\eta$ in the lattice of equivalence relations on a set $S$. (See, e.g.,
Birkhoff~\cite{Birkhoff:1995}, Exercise~10b, page~98.)
Specifically, if the equivalence classes of $\eta$ are $S_1, \dots, S_N$ and if
each $S_j$ $(1\leq j \leq N)$ is the union of $n_j$ equivalence classes of
$\zeta$, then  
\begin{equation}
  \label{eq:basicinterval}
[\zeta, \eta] \cong \prod_{j=1}^N \Eq(n_j).
\end{equation}
Now, if a block of $\tbeta$ contains $C_r$ for some $1\leq r \leq m$, then it
consists of a single block of $\beta^\star$; otherwise, it consists of
$q = |\sT_n\cap \sI_r|$ blocks of $\beta^\star$, say,
$C^{i_1}_{\ell}$, $C^{i_2}_{\ell}$, \dots, $C^{i_q}_{\ell}$,
 for some $1\leq \ell \leq m$, $\ell \neq r$.
Indeed, we have
\begin{itemize}
\item[-] for each $1\leq r \leq m$,
\begin{itemize}
\item[-]  one block of $\tbeta$ consisting of a single block of $\beta^\star$, and
\item[-] for each $1\leq n \leq N$,
\begin{itemize}
\item[] $m-1$ blocks of $\tbeta$ consisting of
 $|\sT_n\cap \sI_r|$ blocks of $\beta^\star$.  
\end{itemize}
\end{itemize}
\end{itemize}
 Therefore, by~(\ref{eq:basicinterval}), we arrive at
 \[
[\beta^\star,\tbeta] \cong
\prod_{r=1}^m \prod_{n=1}^N (\Eq |\sT_n \cap \sI_r|)^{m-1}.\qedhere 
 \]
\end{proof}

We now describe a situation in which the foregoing construction is particularly
useful and easy to apply.
Given an algebra $\bB$ and a pair $(x,y) \in B^2$, the unique smallest congruence
relation of $\bB$ containing $(x,y)$ is called the \emph{principal
  congruence generated by} $(x,y)$, denoted by $\Cg^\bB(x,y)$.
Given a finite congruence lattice  $\Con\bB$, let $\beta = \Cg^\bB(x,y)$, and
consider the overalgebra $\bA$ constructed from  base algebra $\bB$ and
tie-points $T : x,y$.  Then by Theorem~\ref{thm1} the interval consisting of all
$\theta \in \Con\bA$ for which $\theta\resB = \beta$ is given by
$[\beta^*,\widehat{\beta}] \cong \Eq(2)^{m-1} \cong \two^{m-1}$, where $m$ is
the number of congruence classes of $\beta$. Also, since $\beta$ is the unique
smallest congruence containing $(x,y)$, we have  $(x,y)\notin \theta$ for all
$\theta \ngeq \beta$.
Since $x, y$ are the only tie-points, Theorem~\ref{thm1}
implies that for all $\theta \ngeq \beta$ the interval
$[\theta^*,\widehat{\theta}]$ is trivial (since $(x,y)\notin \theta$ implies no
block of $\theta$ contains more than one tie-point). Thus,
$\theta^*=\widehat{\theta}$. It is also immediate from Theorem~\ref{thm1} that
if $\theta \geq \beta$ and if $\theta$ has $r$ congruence
classes, then $[\theta^*,\widehat{\theta}] \cong \two^{r-1}$.

\begin{example}
\label{ex:3.2}
  Theorem~\ref{thm1} explains the shapes of the two congruence
  lattices that we observed in Example~\ref{ex:3.1}.
  Returning to that example, with base algebra $\bB$
  equal to the right regular $S_3$-set, we now show some other congruence lattices that
  result by simply changing the sequence of tie-points.
  Recall that the partitions of $B$ corresponding to nontrivial congruence
  relations of $\bB$ are $\alpha = | 0, 1, 2 | 3, 4, 5|$,
  $\beta = | 0, 3 | 1, 4 | 2, 5 |$,
  $\gamma = | 0, 4 | 1, 5| 2, 3 |$, and
  $\delta = | 0, 5| 1, 3| 2, 4 |$.

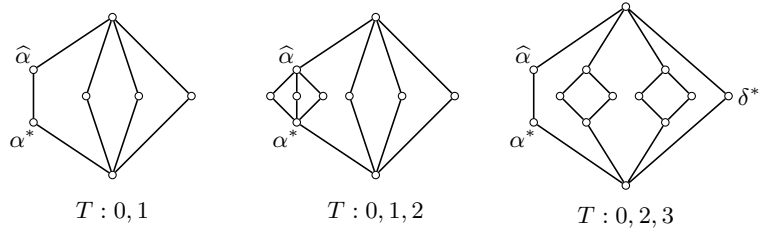
\begin{figure}[h!]
  \centering
    \begin{tikzpicture}[scale=.7]
      \node (150) at (1.5,0)  [draw, circle, inner sep=\dotsize] {};
      \node (01) at (0,1)  [draw, circle, inner sep=\dotsize] {};
      \node (02) at (0,2)  [draw, circle, inner sep=\dotsize] {};
      \node (115) at (1,1.5)  [draw, circle, inner sep=\dotsize] {};
      \node (215) at (2,1.5)  [draw, circle, inner sep=\dotsize] {};
      \node (315) at (3,1.5)  [draw, circle, inner sep=\dotsize] {};
      \node (153) at (1.5,3)  [draw, circle, inner sep=\dotsize] {};
      \draw[semithick]
      (150) to (01) to (02) to (153) to (115) to (150) to (215) to (153) to (315) to (150);
      \draw[font=\small] (1.5,-.7) node {$T : 0,1$};
      \draw[font=\small] (-.2,.7) node {$\alpha^*$};
      \draw[font=\small] (-.2,2.3) node {$\widehat{\alpha}$};

      \node (650) at (6.5,0)  [draw, circle, inner sep=\dotsize] {};
      \node (51) at (5,1)  [draw, circle, inner sep=\dotsize] {};
      \node (52) at (5,2)  [draw, circle, inner sep=\dotsize] {};
      \node (4515) at (4.5,1.5)  [draw, circle, inner sep=\dotsize] {};
      \node (515) at (5,1.5)  [draw, circle, inner sep=\dotsize] {};
      \node (5515) at (5.5,1.5)  [draw, circle, inner sep=\dotsize] {};
      \node (615) at (6,1.5)  [draw, circle, inner sep=\dotsize] {};
      \node (715) at (7,1.5)  [draw, circle, inner sep=\dotsize] {};
      \node (815) at (8,1.5)  [draw, circle, inner sep=\dotsize] {};
      \node (653) at (6.5,3)  [draw, circle, inner sep=\dotsize] {};
      \draw[semithick]
      (650) to (51) to (515) to (52) to (653) to (615) to (650) to (715) to (653) to
      (815) to (650)
      (51) to (4515) to (52) to (5515) to (51);
      \draw[font=\small] (6.5,-.7) node {$T : 0,1,2$};
      \draw[font=\small] (4.8,.7) node {$\alpha^*$};
      \draw[font=\small] (4.8,2.3) node {$\widehat{\alpha}$};

      \node (bot) at (11.25,-.2)  [draw, circle, inner sep=\dotsize] {};
      \node (top) at (11.25,3.2)  [draw, circle, inner sep=\dotsize] {};
      \node (a) at (9.5,1)  [draw, circle, inner sep=\dotsize] {};
      \node (A) at (9.5,2)  [draw, circle, inner sep=\dotsize] {};
      \draw[font=\small] (9.3,.7) node {$\alpha^*$};
      \draw[font=\small] (9.3,2.3) node {$\widehat{\alpha}$};

      \node (b) at (10.5,1)  [draw, circle, inner sep=\dotsize] {};
      \node (b1) at (10,1.5)  [draw, circle, inner sep=\dotsize] {};
      \node (b2) at (11,1.5)  [draw, circle, inner sep=\dotsize] {};
      \node (B) at (10.5,2)  [draw, circle, inner sep=\dotsize] {};

      \node (c) at (12,1)  [draw, circle, inner sep=\dotsize] {};
      \node (c1) at (11.5,1.5)  [draw, circle, inner sep=\dotsize] {};
      \node (c2) at (12.5,1.5)  [draw, circle, inner sep=\dotsize] {};
      \node (C) at (12,2)  [draw, circle, inner sep=\dotsize] {};

      \node (d) at (13.2,1.5)  [draw, circle, inner sep=\dotsize] {};
      \draw[font=\small] (13.6,1.5) node {$\delta^*$};
      \draw[semithick]
      (bot) to (a) to (A) to (top) to (B) to (b1) to (b) to (b2) to (B)
      (b) to (bot) to (c) to (c1) to (C) to (c2) to (c)
      (C) to (top) to (d) to (bot);
      \draw[font=\small] (11.25,-.8) node {$T : 0, 2, 3$};

    \end{tikzpicture}
  \caption{Congruence lattices of overalgebras of the $S_3$-set for various
    tie-point sequences.}
  \label{fig:ConOverAlgebras}
\end{figure}

\begin{figure}[h!]
  \centering
    \begin{tikzpicture}[scale=.7]
      \node (bot) at (3.25,0.5)  [draw, circle, inner sep=\dotsize] {};
      \node (top) at (3.25,4.5)  [draw, circle, inner sep=\dotsize] {};

      \node (a) at (1,2)  [draw, circle, inner sep=\dotsize] {};
      \node (a1) at (.5,2.5)  [draw, circle, inner sep=\dotsize] {};
      \node (a2) at (1,2.5)  [draw, circle, inner sep=\dotsize] {};
      \node (a3) at (1.5,2.5)  [draw, circle, inner sep=\dotsize] {};
      \node (A) at (1,3)  [draw, circle, inner sep=\dotsize] {};
      \draw[font=\small] (.75,1.7) node {$\alpha^*$};
      \draw[font=\small] (.75,3.3) node {$\widehat{\alpha}$};

      \node (b) at (2.5,2)  [draw, circle, inner sep=\dotsize] {};
      \node (b1) at (2,2.5)  [draw, circle, inner sep=\dotsize] {};
      \node (b2) at (3,2.5)  [draw, circle, inner sep=\dotsize] {};
      \node (B) at (2.5,3)  [draw, circle, inner sep=\dotsize] {};

      \node (c) at (4,2)  [draw, circle, inner sep=\dotsize] {};
      \node (c1) at (3.5,2.5)  [draw, circle, inner sep=\dotsize] {};
      \node (c2) at (4.5,2.5)  [draw, circle, inner sep=\dotsize] {};
      \node (C) at (4,3)  [draw, circle, inner sep=\dotsize] {};

      \node (d) at (5.5,2)  [draw, circle, inner sep=\dotsize] {};
      \node (d1) at (5,2.5)  [draw, circle, inner sep=\dotsize] {};
      \node (d2) at (6,2.5)  [draw, circle, inner sep=\dotsize] {};
      \node (D) at (5.5,3)  [draw, circle, inner sep=\dotsize] {};

      \draw[semithick]
      (bot) to (a) to (a1) to (A) to (a2) to (a) to (a3) to (A) to (top) to
      (B) to (b1) to (b) to (b2) to (B)
      (b) to (bot) to (c) to (c1) to (C) to (c2) to (c)
      (C) to (top) to (D) to (d1) to (d) to (d2) to (D)
      (d) to (bot);
      \draw[font=\small] (3.25,-.2) node {$T : 0,1,2,3$};

      \node (Rbot) at (11.25,0.5)  [draw, circle, inner sep=\dotsize] {};
      \node (Rtop) at (11.25,4.5)  [draw, circle, inner sep=\dotsize] {};

      \node (Ra) at (9,2)  [draw, circle, inner sep=\dotsize] {};
      \node (Ra1) at (8.5,2.5)  [draw, circle, inner sep=\dotsize] {};
      \node (Ra2) at (9.5,2.5)  [draw, circle, inner sep=\dotsize] {};
      \node (RA) at (9,3)  [draw, circle, inner sep=\dotsize] {};

      \node (Rb) at (10.5,1.8)  [draw, circle, inner sep=\dotsize] {};
      \node (RB) at (10.5,3.2)  [draw, circle, inner sep=\dotsize] {};
      \draw[font=\small] (10.2,1.5) node {$\beta^*$};
      \draw[font=\small] (10.2,3.4) node {$\widehat{\beta}$};

      \node (Rc) at (12,2)  [draw, circle, inner sep=\dotsize] {};
      \node (Rc1) at (11.5,2.5)  [draw, circle, inner sep=\dotsize] {};
      \node (Rc2) at (12.5,2.5)  [draw, circle, inner sep=\dotsize] {};
      \node (RC) at (12,3)  [draw, circle, inner sep=\dotsize] {};

      \node (Rd) at (13.5,2)  [draw, circle, inner sep=\dotsize] {};
      \node (Rd1) at (13,2.5)  [draw, circle, inner sep=\dotsize] {};
      \node (Rd2) at (14,2.5)  [draw, circle, inner sep=\dotsize] {};
      \node (RD) at (13.5,3)  [draw, circle, inner sep=\dotsize] {};

      \draw[semithick]
      (Rbot) to (Ra) to (Ra1) to (RA) to (Ra2) to (Ra)
      (RA) to (Rtop) to (RB)
      (Rb) to (Rbot) to (Rc) to (Rc1) to (RC) to (Rc2) to (Rc)
      (RC) to (Rtop) to (RD) to (Rd1) to (Rd) to (Rd2) to (RD)
      (Rd) to (Rbot);
      \draw [semithick]
      (Rb) to [out=140,in=-140] (RB)
      (RB) to [out=-40,in=40] (Rb);
      \draw[font=\small] (10.5,2.5) node {$L$};

      \draw[font=\small] (11.25,-.2) node {$T: 0,2,3, 5$};

    \end{tikzpicture}
  \caption{Congruence lattices of overalgebras of the $S_3$-set for various
    tie-point sequences; here, $L\cong \two^2\times\two^2$.}
  \label{fig:ConOverAlgebras2}
\end{figure}
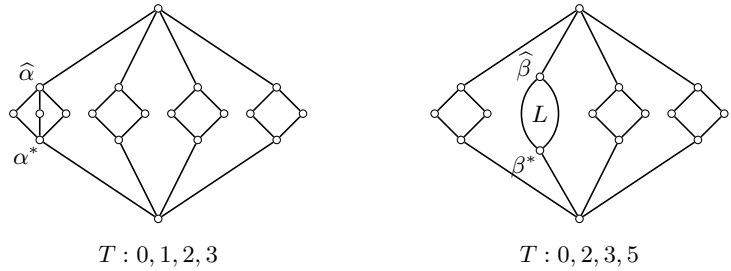

   It is clear from Theorem~\ref{thm1}
 that choosing the sequence $T$ to be $0,1$, or $0,1,2$, or $0, 2, 3$, and taking
 the trivial partition of the tie-point indices ($N=1$),
  yields the congruence lattices appearing in Figure~\ref{fig:ConOverAlgebras}.
When $0, 2, 3, 5$ is chosen
(Figure~\ref{fig:ConOverAlgebras2}, right) we have
$[\beta^*,\widehat{\beta}] \cong \two^2\times\two^2$, which we represent abstractly by $L$
instead of drawing all 16 points in this interval.

\end{example}

Consider again the situation depicted in the congruence lattice on the right of
Figure~\ref{fig:ConOverAlgebras2}, where
$[\beta^*,\widehat{\beta}] = L \cong \two^2\times\two^2$, and suppose we prefer
that all the other $\resB$-inverse images be trivial; that is,
\[
[\beta^*,\widehat{\beta}]\cong \two^2\times\two^2, \ 
\alpha^*=\widehat{\alpha}, \ 
\gamma^*=\widehat{\gamma}, \ 
\delta^*=\widehat{\delta}.
\]
In other words, we seek a finite algebra with a congruence lattice isomorphic to
the lattice in Figure~\ref{fig:ConOverAlgebras3}.
This is easy to achieve by selecting an appropriate partition of the indices
of the tie-points.  Let $t_1, t_2, t_3, t_4 = 0, 2, 3, 5$, and let
the partition of the tie-point indices be $|\sT_1|\sT_2| = |1, 3|2, 4|$.
Then $\beta = |C_1|C_2|C_3| = |0,3|1,4|2,5|$ is the only nontrivial congruence of
$\bB$ having blocks containing multiple tie-points with indices in a single
block of the partition $|\sT_1|\sT_2|$.  Specifically,
$(t_1, t_3) \in \beta$ and $\sT_1\cap \sI_1 = \{1,3\}$ (recall, $\sI_1 = \{i \mid
t_i \in C_1\}$), and  $(t_2, t_4) \in \beta$ and $\sT_2\cap \sI_2 = \{2,4\}$.
Since the number of congruence classes of $\beta$ is $m = 3$, we have
\begin{align*}
[\beta^*, \hbeta] &\cong
\prod_{r=1}^m \prod_{n=1}^N (\Eq |\sT_n \cap \sI_r|)^{m-1}\\
&=
(\Eq|\sT_1\cap \sI_1|)^{m-1} \times
(\Eq|\sT_2\cap \sI_2|)^{m-1} \\
&=
(\Eq 
2)^2 \times (\Eq
2)^2.
\end{align*}

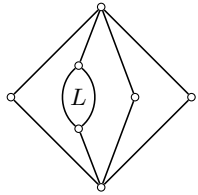
\begin{figure}[h!]
  \centering
    \begin{tikzpicture}[scale=.6]
      \node (Rbot) at (11.25,0.5)  [draw, circle, inner sep=\dotsize] {};
      \node (Rtop) at (11.25,4.5)  [draw, circle, inner sep=\dotsize] {};
      \node (Ra) at (9.25,2.5)  [draw, circle, inner sep=\dotsize] {};
      \node (Rb) at (10.75,1.8)  [draw, circle, inner sep=\dotsize] {};
      \node (RB) at (10.75,3.2)  [draw, circle, inner sep=\dotsize] {};
      \node (Rc) at (12,2.5)  [draw, circle, inner sep=\dotsize] {};
      \node (Rd) at (13.25,2.5)  [draw, circle, inner sep=\dotsize] {};
      \draw[semithick]
      (Rbot) to (Ra) to (Rtop) to (RB)
      (Rb) to (Rbot) to (Rc) to (Rtop) to (Rd) to (Rbot);
      \draw [semithick]
      (Rb) to [out=140,in=-140] (RB)
      (RB) to [out=-40,in=40] (Rb);
      \draw[font=\small] (10.75,2.5) node {$L$};
    \end{tikzpicture}
  \caption{A lattice obtained as a congruence lattice of an overalgebra by
    specifying a suitable partition of the indices of the sequence of tie-points.}
  \label{fig:ConOverAlgebras3}
\end{figure}

A second version of our \GAP\ function used to construct
overalgebras allows the user to specify an arbitrary partition of the
tie-points, and the associated operations $\{s_n \mid 1\leq n \leq N\}$ will be
defined accordingly, as  in (\ref{fn:s_n}).  Continuing with our running example with
tie-points $t_1, t_2, t_3, t_4 = 0,2,3,5$, we now introduce the partition $|1, 3|2,
4|$ of the tie-point indices by invoking the following commands:

\begin{small}
\begin{verbatim}
gap> g:=SymmetricGroup(3);;
gap> B:=[(),(1,2,3),(1,3,2),(1,2),(1,3),(2,3)];;
gap> G:=Action(g,B,OnRight);;
gap> OveralgebraXO([ G, [[0,3], [2,5]] ]);
\end{verbatim}
\end{small}

\noindent 
The resulting overalgebra has a congruence lattice isomorphic to the lattice in
Figure~\ref{fig:ConOverAlgebras3}, with
$L \cong \two^2\times\two^2$.  Similarly,

\begin{small}
\begin{verbatim}
gap> OveralgebraXO([ G, [[0,1,2], [3,4,5]] ]);
\end{verbatim}
\end{small}

\noindent 
produces an overalgebra with congruence lattice isomorphic to the lattice in
Figure~\ref{fig:ConOverAlgebras3}, but with
$L =[\alpha^*, \widehat{\alpha}]\cong \Eq(3) \times\Eq(3)$.  This time $\alpha$ is the only congruence
with nontrivial $\resB$-inverse image.

We conclude this section by noting that, as mentioned at the outset, terms in
the sequence of tie-points may be repeated, and this gives us control over the
number of terms that appear in the product in~(\ref{eq:item-iii}).  For example,
consider the sequence of tie-points $t_1, t_2, \dots, t_9 = 0,1,2,0,1,2,3,4,5$ and
the partition $|1,2,3|4,5,6|7,8,9|$ of the tie-point indices.  The command

\begin{small}
\begin{verbatim}
gap> OveralgebraXO([ G, [[0,1,2], [0,1,2], [3,4,5]] ]);
\end{verbatim}
\end{small}

\noindent 
produces an overalgebra with a 130 element congruence lattice
like the one in Figure~\ref{fig:ConOverAlgebras3}, with
$L =[\alpha^*, \widehat{\alpha}]\cong \Eq(3)\times \Eq(3)\times \Eq(3)$.  Similarly,

\begin{small}
\begin{verbatim}
gap> OveralgebraXO([ G, [[0,3], [0,3], [0,3], [0,3]] ]);
\end{verbatim}
\end{small}

\noindent 
 gives a 261 element congruence lattice
with $ L=[\beta^*, \widehat{\beta}] \cong \two^{16}$. 

Indeed, there is no bound on the number of terms that can be inserted in the
$\resB$-inverse images in $\Con \bA$.  However, as Theorem~\ref{thm1}
makes clear, the shape of each term is invariably a power of a partition
lattice.  It seems the only way to alter this outcome would be
to add more operations to the overalgebra.  Determining how to further expand
an overalgebra in order to achieve such specific goals is one aspect of the
theory that could benefit from further research.

\subsection{Overalgebras II}
\label{sec:overalgebras-ii}
Consider a finite algebra $\bB$ with $\beta\in \Con \bB$ and suppose $\beta$ is
not a principal congruence.
If the method described in Section~\ref{sec:overalgebras-i} is used to construct
an overalgebra $\bA$ in such a way that $[\beta^*, \hbeta]$ is a nontrivial
interval in $\Con\bA$, then there will invariably
be a principal congruence $\theta < \beta$ that also has nontrivial $\resB$-inverse
image  $[\theta^*, \htheta]$.
(This follows from Theorem~\ref{thm1}.)
It is natural to ask whether there
is an alternative overalgebra construction that might result in a
nontrivial interval $[\beta^*, \hbeta]$ such that $\theta^* =  \htheta$ for
all $\theta \ngeq \beta$.
Bill Lampe proposed an ingenious construction to answer this question.
In this section we present a generalization of Lampe's construction
and prove two theorems which describe the resulting congruence lattices.

Let $\bB = \<B, F_B\>$ be a finite algebra, and suppose
\[
\beta = \Cg^{\bB}((a_1, b_1), \dots, (a_{K-1},b_{K-1}))
\]
for some $a_1, \dots, a_{K-1}, b_1, \dots, b_{K-1} \in B$.
Fix an arbitrary integer $u\geq 1$, let $B_1$, $B_2$, \dots, $B_{uK}$ be sets of
cardinality $|B|$, and fix a set of bijections $\pi_i\colon  B\rightarrow B_i$.
As in the Overalgebras I construction, we use the label $x^i$ to denote $\pi_i(x)$, the element of $B_i$
corresponding to $x\in B$ under the bijection $\pi_i$, and we use
$B_0$ to denote $B$.
Arrange the sets so that they intersect as follows (see Figure~\ref{fig:OveralgebrasIII}):
\begin{align*}
B_0\cap B_1 &=\{a_1\}=\{a_1^{1}\},\\
B_1\cap B_2 &=\{b^1_1\}=\{a_2^{2}\},\\
B_2\cap B_3 &=\{b^2_2\}=\{a_3^{3}\}, \dots
\end{align*}
\begin{align*}
\dots, B_{K-2}\cap B_{K-1} &= \{b_{K-2}^{K-2}\}=\{a_{K-1}^{K-1}\},\\
B_{K-1}\cap B_K = B_K\cap B_{K+1}&=\{b^{K-1}_{K-1}\}=\{a^{K}_1\}=\{a^{K+1}_1\},\\
B_{K+1}\cap B_{K+2}&=\{b^{K+1}_{1}\} =\{a^{K+2}_{2}\}, \dots 
\end{align*}
\begin{align*}
\dots, B_{2K-2}\cap B_{2K-1} &= \{b_{K-2}^{2K-2}\}=\{a_{K-1}^{2K-1}\},\\
B_{2K-1}\cap B_{2K} = B_{2K}\cap  B_{2K+1}&=\{b^{2K-1}_{K-1}\}=\{a^{2K}_{1}\}=\{a^{2K+1}_{1}\},\\
B_{2K+1}\cap B_{2K+2}&=\{b^{2K+1}_{1}\} =\{a^{2K+2}_{2}\},\dots 
\end{align*}
\begin{align*}
\dots, B_{uK-2}\cap B_{uK-1} &= \{b^{uK-2}_{K-2}\}=\{a^{uK-1}_{K-1}\},\\
B_{uK-1}\cap B_{uK}&=\{b^{uK-1}_{K-1}\}=\{a^{uK}_{1}\}.
\end{align*}
All other intersections are empty.
In general, for $\ell$ a multiple of $K$,
and for $1\leq i < K$, we have
\begin{align*}
B_{\ell-1}\cap B_{\ell} = B_{\ell}\cap  B_{\ell+1}&=\{b^{\ell-1}_{K-1}\}=\{a^{\ell}_{1}\}=\{a^{\ell+1}_{1}\},\\
B_{\ell+i} \cap
B_{\ell+i+1} &= \{b_i^{\ell+i}\} = \{a_{i+1}^{\ell+i+1}\}.
\end{align*}
\begin{figure}[h!]
  \centering
{\scalefont{.7}
  \begin{tikzpicture}[scale=.52]
    \draw (0, 3.4) node {$B$};
    \draw (0,3) ellipse (.6cm and 1.2cm);
    \draw (1,1.9) node {$B_1$};
    \draw (1,2) ellipse (1.2cm and .6cm);
    \draw (3,1.9) node {$B_2$};
    \draw (3,2) ellipse (1.2cm and .6cm);
   \draw[font=\Large] (5, 2) node {$\cdots$};
    \draw (6.8,1.9) node {$B_{K-2}$};
    \draw (7,2) ellipse (1.2cm and .6cm);
    \draw (9,1.8) node {$B_{K-1}$};
    \draw (9,2) ellipse (1.2cm and .6cm);
     \draw (10,3.4) node {$B_{K}$};
     \draw (10,3) ellipse (.6cm and 1.2cm);
    \draw (11,1.8) node {$B_{K+1}$};
    \draw (11,2) ellipse (1.2cm and .6cm);
    \draw (13.2,1.9) node {$B_{K+2}$};
    \draw (13,2) ellipse (1.2cm and .6cm);
   \draw[font=\Large] (15, 2) node {$\cdots$};
   \node (1) at (0,2) [fill,circle,inner sep=.6pt] {};
   \draw (0, .6) node {$a_1{=}a_1^1$};
   \draw (0, 1) to  (0,1.9);
   \node (2) at (2,2) [fill,circle,inner sep=.6pt] {};
   \draw (2, 3.3) node {$b^1_1 {=} a^2_2$};
   \draw (2, 3) to  (2,2.1);
   \node (3) at (4,2) [fill,circle,inner sep=.6pt] {};
   \draw (4, .6) node {$b^2_2 {=} a^3_3$};
   \draw (4, 1) to  (4,1.9);
   \node (5) at (8,2) [fill,circle,inner sep=.6pt] {};
   \draw (7.5, 3.4) node {$b^{K-2}_{K-2} {=} a^{K-1}_{K-1}$};
    \draw (8, 3) to  (8,2.1);
   \node (6) at (10,2) [fill,circle,inner sep=.6pt] {};
   \draw (10, .6) node {$b^{K-1}_{K-1} {=} a^{K}_{1}{=} a^{K+1}_{1}$};
   \draw (10, 1) to  (10,1.9);
   \node (7) at (12,2) [fill,circle,inner sep=.6pt] {};
   \draw (12.5, 3.4) node {$b^{K+1}_{1} {=} a^{K+2}_{2}$};
   \draw (12, 3) to  (12,2.1);
   \draw[font=\Large] (-.3, 2-5) node {$\cdots$};
    \draw (15-13.5,1.9-5) node {$B_{2K-2}$};
    \draw (15-13.3,2-5) ellipse (1.35cm and .6cm);
    \draw (16.9-13,1.8-5) node {$B_{2K-1}$};
    \draw (17-13.1,2-5) ellipse (1.35cm and .6cm);
    \draw (18-13,3.3-5) node {$B_{2K}$};
    \draw (18-13,3-5) ellipse (.6cm and 1.2cm);
    \draw (19.1-12.9,1.8-5) node {$B_{2K+1}$};
    \draw (19-12.9,2-5) ellipse (1.35cm and .6cm);
    \draw (19.1-10.6,1.9-5) node {$B_{2K+2}$};
    \draw (19-10.6,2-5) ellipse (1.35cm and .6cm);
   \draw[font=\Large] (21-10.6, 2-5) node {$\cdots$};
   \node (9) at (16-13.2,2-5) [fill,circle,inner sep=.6pt] {};
   \draw (15.5-13.2, 3.4-5) node {$b^{2K-2}_{K-2}{=}a^{2K-1}_{K-1}$};
   \draw (15.5-13.2, 3-5) to  (16-13.2,2.1-5);
   \node (10) at (18-13,2-5) [fill,circle,inner sep=.6pt] {};
   \draw (18-13, .6-5) node {$b^{2K-1}_{K-1} {=} a^{2K}_{1}{=} a^{2K+1}_{1}$};
   \draw (18-13, 1-5) to  (18-13,1.9-5);
   \node (11) at (20-12.8,2-5) [fill,circle,inner sep=.6pt] {};
   \draw (20-12.2, 3.4-5) node {$b^{2K+1}_{1}{=}a^{2K+2}_{2}$};
   \draw (20-12.5, 3-5) to  (20-12.8,2.1-5);

   \draw[font=\Large] (13-.2, 2-5) node {$\cdots$};
    \draw (14.8-.3,1.9-5) node {$B_{uK-2}$};
    \draw (15-.3,2-5) ellipse (1.35cm and .6cm);
    \draw (16.9,1.8-5) node {$B_{uK-1}$};
    \draw (17-.1,2-5) ellipse (1.35cm and .6cm);
    \draw (18,3.3-5) node {$B_{uK}$};
    \draw (18,3-5) ellipse (.6cm and 1.2cm);
   \node (9) at (16-.2,2-5) [fill,circle,inner sep=.6pt] {};
   \draw (15.5-.2, 3.4-5) node {$b^{uK-2}_{K-2}{=}a^{uK-1}_{K-1}$};
   \draw (15.5-.2, 3-5) to  (16-.2,2.1-5);
   \node (10) at (18,2-5) [fill,circle,inner sep=.6pt] {};
   \draw (18, .6-5) node {$b^{uK-1}_{K-1} {=} a^{uK}_{1}$};
   \draw (18, 1-5) to  (18,1.9-5);
  \end{tikzpicture}
}
  \caption{The universe of an overalgebra of the second type.}
\label{fig:OveralgebrasIII}
\end{figure}
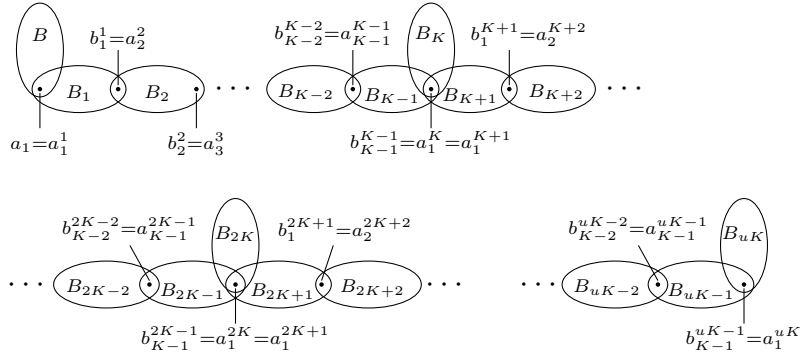

As usual, we put $A\defeq B_0\cup \dots\cup B_{uK}$ and proceed to define
some unary operations on $A$.
First, for $0\leq i, j \leq uK$, let $S_{i,j}\colon B_i \rightarrow B_j $ be the
bijection $S_{i,j}(x\supi)=x\supj$, and note that $S_{i,i} = \id_{B_i}$.
Let $\sT = |\sT_1|\sT_2| \cdots|\sT_N|$ be an arbitrarily chosen partition of the set
$\{0, K, 2K, \dots, uK\}$.  For each $1\leq n \leq N$, for each $\ell \in \sT_n$, define
\[
e_{\ell}(x)=
\begin{cases}
S_{j,\ell}(x), &\text{ if $x\in B_{j}$ for some $j \in \sT_n$,}\\
a^{\ell}_1, &\text{ otherwise.}
\end{cases}
\]
For each
 $\ell \in \{0, K, 2K, \dots, (u-1)K\}$, for each
$1\leq i < K$, define
\[
e_{\ell+i}(x)=
\begin{cases}
a_i^{\ell+i}, &\text{ if $x\in B_{j}$ for some $j < \ell+i$,}\\
x, &\text{ if $x\in B_{\ell+i}$,}\\
b_i^{\ell+i}, &\text{ if $x\in B_{j}$ for some $j > \ell+i$.}
\end{cases}
\]
In other words, if $\ell$ is a multiple of $K$, then $e_{\ell}$ maps each up-pointing set in
Figure~\ref{fig:OveralgebrasIII} in the same $\sT$-block  as $\ell$
bijectively onto the up-pointing set $B_{\ell}$, and maps all other points of $A$ to the tie-point
$a^{\ell}_1\in B_{\ell}$.
For each set $B_{\ell+i}$ in between---represented in the
figure by an ellipse with horizontal major axis---there is a map
$e_{\ell+i}$ which act as the identity on $B_{\ell+i}$ and maps all points in
$A$ left of $B_{\ell+i}$ to the left tie-point 
of $B_{\ell+i}$ and all points to
the right of $B_{\ell+i}$ to the right tie-point 
of $B_{\ell+i}$.
Finally, for $0\leq i, j\leq uK$, we define
$q_{i,j}=S_{i,j}\circ e_i$ and take the set of basic operations on $A$ to be
\[
F_A \defeq  \{f e_0 \suchthat f\in F_B\}  \cup \{q_{i,0} \suchthat 0\leq i \leq uK\}\cup \{q_{0,j} \suchthat 1\leq j \leq uK\}.
\]
We are now ready to define the overalgebra as $\bA \defeq  \< A, F_A\>$.

In this overalgebra construction the significance of a particular congruence of
$\bB$---namely,
$\beta = \Cg^{\bB}((a_1, b_1), \dots, (a_{K-1},b_{K-1}))$---is
more explicit than in the overalgebra construction of
Section~\ref{sec:overalgebras-i}.  The following theorem describes the
$\resB$-inverse image of this special congruence, that is, the interval
$[\beta^*,\hbeta]$ in $\Con \bA$.   As above, assume $\beta$ has $m$
congruence classes, denoted by 
$C_r$ $(1\leq r\leq m)$,
let $C_r^j$ denote $S_{0,j}(C_r)$,
and let $\beta^j$ denote $S_{0,j}(\beta)$.
\begin{theorem}
\label{thm:overalgebras-ii}
Let $\bA = \< A, F_A\>$ be the overalgebra described above,
and for each $0\leq j \leq uK$ let $t_j$ denote a tie-point of the set
$B_j$.  Define
\begin{equation}
  \label{eq:OA2-star}
\beta^\star = \bigcup_{j=0}^{uK} \betaj \cup
\bigl(\bigcup_{j=0}^{uK}t_j/\betaj\bigr)^2.
\end{equation}
and
\begin{equation}
  \label{eq:5}
\tbeta   = \beta^\star \cup \bigcup_{r=1}^{m-1}\bigcup_{n=1}^N \bigl(\bigcup_{\ell \in \sT_n} C_r^\ell\bigr)^2
\end{equation}
Then,
\begin{enumerate}[\rm(i)] 
\item \label{item:OA2-i}
$\beta^\star=\beta^*$, the minimal $\theta\in \Con\bA$ such that $\theta\resB = \beta$;
\item
\label{item:OA2-ii}
$\tbeta=\widehat{\beta}$, the maximal $\theta \in \Con\bA$ such that $\theta\resB = \beta$;
\item
\label{item:OA2-iii}
the interval $[\beta^\star, \tbeta]$ in $\Con\bA$
satisfies
$[\beta^\star, \tbeta] \cong \prod_{n=1}^N(\Eq|\sT_n|)^{m-1}$.
\end{enumerate}
\end{theorem}

\begin{proof}
 (i) It is easy to see that $\beta^\star$ is an equivalence relation on $A$, so we first
  check that $f(\beta^\star)\subseteq \beta^\star$ for all $f\in F_A$, thereby
  establishing that $\beta^\star\in \Con\bA$. Then we show that
  $\beta \subseteq \eta \in \Con\bA$ implies $\beta^\star\leq \eta$.

For all $0\leq i,j \leq uK$ and $1\leq r \leq m$, either $e_i(C_r^j)$ is a
singleton or, in case $i$ and $j$ are multiples of $K$ in the same
block of $\sT$, we have
$e_i(C_r^j) = S_{j,i}(C_r^j) = C_r^i$.  Therefore,
$q_{i,0}(C_r^j) =S_{i,0}e_i(C_r^j)$ is
either a singleton or $C_r$, so
$q_{i,0}(\bigcup_{j=0}^{uK}\beta^j) \subseteq \beta$.
Similarly, since $e_0(C_r^j)$ is either $\{a_1\}$ or $C_r$, we see that
$q_{0,i}(C_r^j) = S_{0,i}e_0(C_r^j)$ is either $\{a_1^i\}$ or $C_r^i$. Also, for each $g\in F_B$
we have $ge_0(C_r^j)\subseteq C_k$ for some $0\leq k\leq m$.
Therefore, $f(\bigcup_{j=0}^{uK}\beta^j)\subseteq \beta^{\star}$ for all $f\in F_A$.

Let $\sB \defeq  \bigcup_{j=0}^{uK}t_j/\betaj$.
Then for each $0\leq i \leq uK$ we have $e_i(\sB) = t_i/\betai$, so
$q_{i,0}(\sB) = S_{i,0}(t_i/\betai) = C_k$, for some $1\leq k \leq m$.
Since $e_0(\sB) = a_1/\beta = C_1$, it is clear that $q_{0,j}(\sB) = S_{0,j}e_0(\sB) =
C_1^j$.  Also, for each $g\in F_B$, we have $g e_0(\sB) = g(C_1) \subseteq C_k$
for some $1\leq k\leq m$.  Therefore, $f(\sB)^2 \subseteq \beta^\star$ for all
$f\in F_A$.  We conclude that $\beta^\star \in \Con \bA$.

Next, suppose $\beta\subseteq \eta \in \Con \bA$.  Then for all $0\leq j \leq
uK$ we have $q_{0,j}(\beta) = S_{0,j}e_0(\beta) = \beta^j \subseteq \eta$.
For each
$\ell \in \{0, K, 2K, \dots, (u-1)K\}$ and $0\leq i <K$, the tie-points of
$B_{\ell+i}$ are $a_1^\ell$ if $i=0$, and $\{a_i^{\ell+i}, b_i^{\ell+i}\}$ if
$i>0$.  Thus, from the fact that $(a_i^{\ell+i}, b_i^{\ell+i})\in
\beta^{\ell+i}$ and from the overlapping of $B_{\ell+i}$ and
$B_{\ell+i+1}$ it follows
by transitivity that any equivalence relation on $A$ containing all $\beta^j$, $0\leq
j \leq uK$, must have a single equivalence class containing all the tie-points of $A$.
We have just seen that $\eta$ contains all $\beta^j$, so
$\bigl(\bigcup_{j=0}^{uK}t_j/\betaj\bigr)^2\subseteq \eta$.
Therefore, $\beta^\star \leq \eta$.

\smallskip 
(ii) We first show $\tbeta\in \Con\bA$.
Fix $1\leq r \leq m$ and $1\leq n\leq N$ and let
\[
\CE\defeq  \bigcup_{\ell \in \sT_n} C_r^\ell.
\]
Thus, $\CE$ is the join of corresponding 
$\beta$ blocks of
up-pointing sets in Figure~\ref{fig:OveralgebrasIII} from the same block
$\sT_n$.

If $\ell \in \{0, K, 2K, \dots, (u-1)K\}$ and $1\leq i< K$, then $B_{\ell+i}$ has tie-points
$\{a_i^{\ell+i}, b_i^{\ell+i}\}$, and the mapping $e_{\ell+i}$ takes the set
$\CE$ onto this pair of tie-points.
Thus, if $\ell \in \{0, K, 2K, \dots, (u-1)K\}$, $1\leq i< K$, and
$k=\ell+i$, then
$q_{k,0}(\CE) = \{a_i, b_i\}$.
It follows that $q_{k,0}(\tbeta)\subseteq \beta$ for all
$k\notin \{0, K, 2K, \dots, uK\}$.

If $k\in \{0, K, 2K, \dots, uK\}\setminus \sT_n$, then
$e_k(\CE) = \{a_1^k\}$ so
$q_{k,0}(\CE) = \{a_1\}$. If
$k\in \sT_n$, then $e_k(\CE) = C_r^k$ so $q_{k,0}(\CE) = C_r$.  It follows that
$q_{k,0}(\tbeta)\subseteq \beta$ for all $0\leq k \leq uK$.

Next, if $g\in F_B$ then, since $e_{0}(\CE)$ is either $C_r$ or
$\{a_1\}$, the operation
$ge_0$ takes $\CE$ to a single $\beta$ class.
Finally, for each $0 \leq k \leq uK$ the map
$q_{0,k}$ takes the set $\CE$ to either $C_r^k$ or $\{a_1^k\}$.
Therefore, $q_{0,k}(\tbeta)\subseteq \tbeta$, and we have proved
$f(\tbeta)\subseteq \tbeta$ for all $f\in F_A$.

Since the
restriction of $\tbeta$ to $B$ is clearly $\tbeta \resB = \beta$, the
residuation lemma yields $\tbeta \leq \hbeta$.  On the other hand, it is
easy to verify that for each $(x,y)\notin
\tbeta$ there is an operation $f\in \Pol_1(\bA)$ such that $(e_0f(x),
e_0f(y))\notin \beta$, and thus $(x,y)\notin \hbeta$.  Therefore, $\tbeta \geq \hbeta$.

\smallskip 
(iii) It remains to prove
$[\beta^\star, \tbeta] \cong \prod_{n=1}^N(\Eq|\sT_n|)^{m-1}$.
This follows easily from the proof of (\ref{item:OA2-ii}).  For,
in proving that $\tbeta$ is a congruence, we showed that each
operation $f\in F_A$ maps blocks of $\tbeta$ into blocks of
$\beta^\star$.  That is, each operation collapses the interval $[\beta^\star,
\tbeta]$.  Therefore, every equivalence relation on the set $A$ that lies
between $\beta^\star$ and $\tbeta$ is respected by every operation of $\bA$. In
other words, as an interval in $\Con \bA$,
\[
[\beta^\star, \tbeta] = \{\theta \in \Eq(A)\suchthat \beta^\star \leq \theta \leq \tbeta\}.
\]
In view of the configuration of the universe of $\bA$,
as shown in Figure~\ref{fig:OveralgebrasIII}, and by the same argument used to prove
(\ref{item-iii}) of Theorem~\ref{thm1} (cf.~(\ref{eq:basicinterval})),
it is clear that the interval sublattice
$\{\theta \in \Eq(A)\suchthat \beta^\star \leq \theta \leq \tbeta\}$
is isomorphic to $\prod_{n=1}^N(\Eq|\sT_n|)^{m-1}$.
\end{proof}

In the next theorem, we continue to assume that $\bA = \< A, F_A\>$ is an
overalgebra of the second type---as illustrated in
Figure~\ref{fig:OveralgebrasIII}---based on the algebra $\bB$, the congruence
$\beta = \Cg^{\bB}((a_1, b_1), \dots, (a_{K-1},b_{K-1}))$, and the partition
$\sT = |\sT_1|\sT_2| \cdots |\sT_N|$ of the set  $\{0, K, 2K, \dots, uK\}$.
We remind the reader that $\theta^*$ denotes $\Cg^\bA(\theta)$, for
$\theta\in \Con \bB$.

\begin{theorem}
\label{thm3}
Let $\theta\in \Con \bB$ and suppose $\theta$ has $r$ congruence classes.
Then, $\theta^* < \widehat{\theta}$ if and only if
$\beta\leq \theta < 1_B$, in which case $[\theta^*, \widehat{\theta}] \cong
\prod_{n=1}^N(\Eq|\sT_n|)^{r-1}$.
\end{theorem}
Consequently, if $\theta \ngeq \beta$, then $\widehat{\theta} = \theta^*$.
The proof of the theorem follows easily from the next lemma.
\begin{lemma}
\label{lem3.1}
Suppose $\eta \in \Con\bA$ satisfies $\eta\resB = \theta \in \Con \bB$ and $(x,y) \in \eta \setminus \theta^*$ for some
$x\in B_{i}$,  and $y\in B_{j}$.  Then $i$ and $j$ are distinct multiples of $K$ belonging to the same
block of $\sT$, and $\theta \geq \beta$.
\end{lemma}
\begin{proof}
Assume $\eta \in \Con\bA$ satisfies $\eta\resB = \theta$ and $(x,y) \in \eta \setminus \theta^*$ for some
$x\in B_i, y\in B_j$.
If $i=j$, then $(x,y)\in \eta \cap B_i^2$ and
$(q_{i,0}(x),q_{i,0}(y))\in \eta\resB = \theta \leq \theta^*$, so
$(x,y) = (q_{0,i} q_{i,0}(x),q_{0,i} q_{i,0}(y))\in \theta^*$, a contradiction.
In fact, we will reach this same contradiction, $(x,y)\in  \theta^*$, as long as
$i$ and $j$ do not belong to the same block of $\sT$.

We have already handled the case $i=j$, so we may suppose $0\leq i < j \leq uK$.
In order to make the simple idea of the proof more transparent, we first
consider the special case in which $1\leq i < j < K$.
In this case, we have
$q_{i,0}(y)= b_{i}$, so
$(x,y)\in \eta$ implies
$(q_{i,0}(x),q_{i,0}(y)) =
(q_{i,0}(x), b_{i}) \in \eta\resB = \theta$.  Similarly,
$q_{j,0}(x)= a_{j}$, so
$(q_{j,0}(x),q_{j,0}(y)) = (a_{j},q_{j,0}(y)) \in \theta$.
In case $j= i+1$, we obtain 
\begin{equation}
  \label{eq:4}
x = q_{0,i}q_{i,0}(x) \mathrel{\theta^*} q_{0,i}(b_i) = b_i^i = a^j_j = q_{0,j}(a_j) \mathrel{\theta^*} q_{0,j} q_{j,0}(y) = y,
\end{equation}
so $(x,y)\in \theta^*$.  The relations~(\ref{eq:4}) are illustrated by the following diagram:
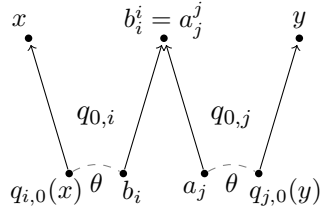
\begin{center}
     \begin{tikzpicture}[scale=.6]
      \node (00) at (-.1,0) [fill,circle,inner sep=1pt] {};
      \node (10) at (1.1,0) [fill,circle,inner sep=1pt] {};
      \node (30) at (2.9,0) [fill,circle,inner sep=1pt] {};
      \node (40) at (4.1,0) [fill,circle,inner sep=1pt] {};
      \draw (-0.6,-.4) node {$q_{i,0}(x)$};
      \draw (1.3,-.4) node {$b_i$};
      \draw (.5,-.2) node {$\theta$};
      \draw (2.7,-.4) node {$a_j$};
      \draw (3.5,-.2) node {$\theta$};
      \draw (4.7,-.4) node {$q_{j,0}(y)$};
      \node (m13) at (-1,3) [fill,circle,inner sep=1pt] {};
      \node (23) at (2,3) [fill,circle,inner sep=1pt] {};
      \node (53) at (5,3) [fill,circle,inner sep=1pt] {};
      \draw (-1.2,3.4) node {$x$};
      \draw (2,3.4) node {$b_i^i = a_j^j$};
      \draw (5,3.4) node {$y$};
       \draw[font=\large] (.5,1.25) node {$q_{0,i}$};
       \draw[font=\large] (3.5,1.25) node {$q_{0,j}$};
      \path[->] (00) edge (-.95,2.85);
      \path[->] (10) edge (1.95,2.85);
      \path[->] (30) edge (2.05,2.85);
      \path[->] (40) edge (4.95,2.85);
      \draw[dashed, gray]
      (00) to [out=30,in=150] (10)
      (30) to [out=30,in=150] (40);
    \end{tikzpicture}
\end{center}
\noindent 
In case $j> i+1$, we obtain
the diagram below.
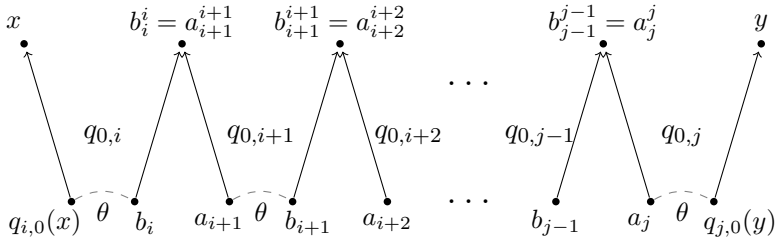
\begin{center}

\begin{tikzpicture}[scale=.7]
  \node (00) at (-.1,0) [fill,circle,inner sep=1pt] {};
  \node (10) at (1.1,0) [fill,circle,inner sep=1pt] {};
  \draw (-0.6,-.4) node {$q_{i,0}(x)$};
  \draw (.5,-.2) node {$\theta$};
  \draw (1.3,-.4) node {$b_i$};

  \node (30) at (2.9,0) [fill,circle,inner sep=1pt] {};
  \node (40) at (4.1,0) [fill,circle,inner sep=1pt] {};
  \draw (2.7,-.4) node {$a_{i+1}$};
  \draw (3.5,-.2) node {$\theta$};
  \draw (4.4,-.4) node {$b_{i+1}$};

  \node (60) at (5.9,0) [fill,circle,inner sep=1pt] {};
  \draw (5.9,-.4) node {$a_{i+2}$};
  \node (90) at (9.1,0) [fill,circle,inner sep=1pt] {};
  \draw (9.1,-.4) node {$b_{j-1}$};
  \node (110) at (10.9,0) [fill,circle,inner sep=1pt] {};
  \node (120) at (12.1,0) [fill,circle,inner sep=1pt] {};
  \draw (10.7,-.4) node {$a_j$};
  \draw (11.5,-.2) node {$\theta$};
  \draw (12.6,-.4) node {$q_{j,0}(y)$};
  \node (103) at (10,3) [fill,circle,inner sep=1pt] {};
  \node (133) at (13,3) [fill,circle,inner sep=1pt] {};
  \draw (10,3.4) node {$b_{j-1}^{j-1} = a_{j}^{j}$};
  \node (m13) at (-1,3) [fill,circle,inner sep=1pt] {};
  \node (23) at (2,3) [fill,circle,inner sep=1pt] {};
  \node (53) at (5,3) [fill,circle,inner sep=1pt] {};
  \draw (-1.2,3.4) node {$x$};
  \draw (2,3.4) node {$b_i^i = a_{i+1}^{i+1}$};
  \draw (5,3.4) node {$b_{i+1}^{i+1} = a_{i+2}^{i+2}$};
  \draw (13,3.4) node {$y$};
  \path[->] (00) edge (-.95,2.85);       \draw[font=\large] (.5,1.25) node {$q_{0,i}$};
  \path[->] (10) edge (1.95,2.85);
  \path[->] (30) edge (2.05,2.85);       \draw[font=\large] (3.5,1.25) node {$q_{0,i+1}$};
  \path[->] (40) edge (4.95,2.85);
  \path[->] (60) edge (5.05,2.85);       \draw[font=\large] (6.3,1.25) node {$q_{0,i+2}$};
                                         \draw [font=\LARGE] (7.5,2.25) node {$\dots$};
                                         \draw [font=\LARGE] (7.5,0) node {$\dots$};
  \path[->] (90) edge (9.95,2.85);       \draw[font=\large] (8.8,1.25) node {$q_{0,j-1}$};
  \path[->] (110) edge (10.05,2.85);
  \path[->] (120) edge (12.95,2.85);       \draw[font=\large] (11.5,1.25) node {$q_{0,j}$};
  \draw[dashed, gray]
   (00) to [out=30,in=150] (10)
   (30) to [out=30,in=150] (40)
   (110) to [out=30,in=150] (120);

\end{tikzpicture}
\end{center}
Here too we could write out a line analogous to~(\ref{eq:4}), but it is
obvious from the diagram that $(x,y)\in \theta^*$.

To handle the general case, we let $i = vK+p$ and $j=wK+q$, for some $0\leq v \leq w \leq u$ and $0\leq p, q <
K$.  Assume for now that $p, q \geq 1$, so that neither $i$ nor $j$ is a
multiple of $K$.  Then
$q_{i,0}(y)= q_{vK+p,0}(y)=
b_{p}$ so $(q_{i,0}(x),q_{i,0}(y)) =
(q_{i,0}(x), b_{p})$ belongs to
$\eta\resB = \theta$.  Similarly,
$q_{j,0}(x)=
a_{q}$ so
$(q_{i,0}(x),q_{i,0}(y)) =
(a_q, q_{j,0}(y)) \in \theta$.
In this case, we have
\begin{align}
\label{eq:thetastar}
  x=q_{0,i}q_{i,0}(x) & \mathrel{\theta^*} q_{0,i}(b_p) = b_p^{vK+p} = a^{vK+p+1}_{p+1} = q_{0,i+1}(a_{p+1})\nonumber\\
  & \mathrel{\theta^*} q_{0,i+1} (b_{p+1})= b^{vK+p+1}_{p+1} = a^{vK+p+2}_{p+2} = q_{0,i+2}(a_{p+2})\nonumber\\
  &\ \vdots\\ 
  & \mathrel{\theta^*} q_{0,(v+1)K-1} (b_{K-1}) = b^{(v+1)K-1}_{K-1} = a^{(v+1)K}_1.\nonumber
\end{align}
Now $a^{(v+1)K}_1=a^{(v+1)K+1}_1=q_{0,(v+1)K+1}(a_1)$, so the relations in (\ref{eq:thetastar}) imply
$x \mathrel{\theta^*} q_{0,(v+1)K+1}(a_1)$.  Continuing, we have
\begin{align*}
  x &\mathrel{\theta^*}q_{0,(v+1)K+1}(a_1)\\
 &\mathrel{\theta^*} q_{0,(v+1)K+1}(b_1) = b^{(v+1)K+1}_{1} = a^{(v+1)K+2}_2=q_{0,(v+1)K+1}(a_2) \\
  &\ \vdots\\ 
  &\mathrel{\theta^*}q_{0,wK+q-1}(b_{q-1}) = b^{wK+q-1}_{q-1} = a^{wK+q}_q=q_{0,j}(a_q)\\
  &\mathrel{\theta^*} q_{0,j} q_{j,0}(y) = y.
\end{align*}
Again, we arrive at the desired contradiction, $(x,y)\in \theta^*$.

The case in which exactly one of $i, j$ is a multiple of $K$ is almost
identical, so we omit the derivation.
Finally, suppose $i$ and $j$ are distinct multiples of $K$.  If $i$ and $j$
belong to distinct blocks of the partition $\sT$, then
$(q_{i,0}(x), q_{i,0}(y)) = (q_{i,0}(x), a_1) \in \theta$ and
$(q_{j,0}(x), q_{j,0}(y)) = (a_1, q_{j,0}(y)) \in \theta$, and the
foregoing argument will again lead to the contradiction
$(x,y)\in \theta^*$.

The only remaining possibility is $i, j \in \sT_n$ for some $1\leq n\leq N$.
In this case, there is no contradiction.
Rather, we observe that
(as in a number of the contradictory cases)
there is at least one $v \in
\{0, 1, \dots,u-1\}$ such that, for all
$1\leq k < K$, we have $i< vK+k < j$ and thus
$(q_{vK+k,0}(x),q_{vK+k,0}(y)) = (a_k, b_k) \in \eta\resB = \theta$.
Therefore,
 $\theta \geq \beta = \Cg^{\bA}((a_1, b_1), \dots, (a_{K-1}, b_{K-1}))$.
\end{proof}

\begin{proof}[Proof of Theorem~\ref{thm3}]
Lemma~\ref{lem3.1} implies that $\theta^* < \widehat{\theta}$ only if
$\beta\leq \theta < 1_B$. On the other hand,
if $\beta\leq \theta < 1_B$, then we obtain
$[\theta^*, \widehat{\theta}] \cong \prod_{n=1}^N(\Eq|\sT_n|)^{r-1}$
by the argument used to prove the same fact about
$[\beta^\star, \tbeta]$
in Theorem~\ref{thm:overalgebras-ii}.
\end{proof}

\section{Conclusion}
We have described an approach to building new finite algebras out of old
which is useful in the following situation: given an algebra $\bB$,
we seek an algebra $\bA$ with congruence lattice $\Con \bA$ such that $\Con \bB$
is a (non-trivial) homomorphic image of $\Con \bA$; specifically, we construct
$\bA$ so that $\resB \colon  \Con \bA \rightarrow \Con \bB$ is a lattice epimorphism.
We described the original example---the ``triple-winged pentagon'' shown on
the right of Figure~\ref{fig:sevens}---found by Ralph Freese, which
motivated us to develop a general procedure for finding such finite algebraic
representations.

We mainly focused on two specific overalgebra constructions.
In each case, the congruence lattice that results has the same basic shape as
the one with which we started, except that some congruences are replaced with
intervals that are direct products of powers of partition lattices.
Thus we have identified a broad new class of finitely representable lattices.
However, the fact that the new intervals in these lattices must be products of
partition lattices seems quite limiting, and this is the first limitation that we
think future research might aim to overcome.

We envision variations on the constructions described in this paper
which might bring us closer toward the goal of replacing certain congruences,
$\beta\in \Con \bB$, with more general finite lattices,
$L\cong [\beta^*, \widehat{\beta}] \leq \Con \bA$.
However, using the constructions described above, we have found examples of
overalgebras for which it is not possible to simply add operations
in order to eliminate \emph{all} relations $\theta$ such that
$\beta^* < \theta < \widehat{\beta}$.
Nonetheless, we remain encouraged by some modest progress we have
recently made in this direction.

As a final remark, we call attention to another obvious limitation of
the methods described in this paper---they cannot be used to find an
algebra with a \emph{simple} congruence lattice.  For example, the lattice on
the left in Figure~\ref{fig:sevens}---the
``winged $\two \times \three$,''  for lack of a better name---is a simple
lattice, so it is certainly not the $\resB$-inverse image of some
smaller lattice.
As of this writing, we do not know of a finite algebra that has
a winged $\two \times \three$ congruence lattice.  More details about this
aspect of the problem appear in~\cite{IEProps}. 

\subsection*{Acknowledgments}
  A number of people contributed to the present work.  
  In particular, the main ideas were develop jointly
  by the Ralph Freese, Peter Jipsen, Bill Lampe, J.B.~Nation and the author.
  The primary referee made many insightful recommendations
  which led to vast improvements, not only in the overall exposition, but also in
  the specific algebraic constructions, theorem statements, and proofs.
  This article is dedicated to Ralph Freese and Bill Lampe, who
  declined to be named as co-authors despite their substantial contributions to the project.

\end{document}